\documentclass{article}
\usepackage{graphicx, float}
\usepackage[english]{babel}
\usepackage{caption}
\usepackage{subcaption}
\usepackage{amsmath, amsfonts, amssymb, amsthm, bm}
\usepackage{mathrsfs}
\usepackage{url}
\usepackage{tabularray}
\usepackage{xcolor, colortbl}
\definecolor{Gray}{gray}{0.95}
\usepackage{hyperref}

\usepackage{booktabs}

\theoremstyle{plain}
\newtheorem{theorem}{Theorem}[section]

\newtheorem{lemma}[theorem]{Lemma}
\newtheorem{proposition}[theorem]{Proposition}

\theoremstyle{definition}

\usepackage{mathtools}
\newcommand{\bydef}{\coloneqq}
\renewcommand{\Re}{\mathrm{Re}}

\usepackage[mathlines]{lineno}

\begin{document}
\title{
\textbf{
Global Continuation of Stable Periodic Orbits \\ in Systems of Competing Predators
}}
\author{
Kevin E. M. Church
\thanks
{Université de Montréal, Centre de Recherches Mathématiques, 
Canada. {\tt kevin.church@umontreal.ca}}
\and Jia-Yuan Dai
\thanks
{National Tsing Hua University, Department of Mathematics, 
Taiwan; 
National Center for Theoretical Sciences,
National Taiwan University, Taiwan. {\tt jydai@math.nthu.edu.tw}}
\and Olivier Hénot
\thanks
{National Taiwan University, Department of Mathematics, 
Taiwan. {\tt olivierhenot@ntu.edu.tw}.}
\and Phillipo Lappicy
\thanks
{Universidad Complutense de Madrid, Departamento de Análisis Matemático y Matemática Aplicada, 
Spain; 
Instituto de Ciencias Matemáticas (ICMAT), CSIC-UAM-UC3M-UCM, Spain. {\tt philemos@ucm.es}}
\and Nicola Vassena
\thanks
{Universität Leipzig, Interdisziplinäres Zentrum für Bioinformatik, 
Germany. {\tt nicola.vassena@uni-leipzig.de}}
}

\date{}

\maketitle

\begin{abstract}
We develop a continuation technique to obtain global families of stable periodic orbits, delimited by transcritical bifurcations at both ends.
We formulate a zero-finding problem whose zeros correspond to families of periodic orbits.
We then define a Newton-like fixed-point operator and establish its contraction near a numerically computed approximation of the family.
To verify the contraction, we derive sufficient conditions expressed as inequalities on the norms of the fixed-point operator, and involving the numerical approximation.
These inequalities are then rigorously checked by the computer via interval arithmetic.
To show the efficacy of our approach, we prove the existence of global families in an ecosystem with Holling's type II functional response, and thereby solve a stable connection problem proposed by Butler and Waltman in 1981.
Our method does not rely on restricting the choice of parameters and is applicable to many other systems that numerically exhibit global families.
\end{abstract}

\begin{center}
{\bf \small Key words:}
{\small periodic orbits, stability, continuation, computer-assisted proofs, \\ Newton--Kantorovich theorem, predator-prey systems}
\end{center}




\section{Introduction}
\label{sec:introduction}

{Understanding stable coexistence in ecological systems remains a central challenge in mathematical biology. To address this, we present an effective continuation method, based on computer-assisted proofs, to obtain global families of stable periodic orbits undergoing transcritical bifurcations at invariant boundary planes. This framework provides a general tool to study the \emph{competitive exclusion principle} and to rigorously explore mechanisms of stable coexistence; see  \cite{gause1934, MathBioBook}.
\\
\indent More specifically,} our interest in this question originated from studying the following system (see \cite{10.1137/0135051, HHW1978}) that describes the population dynamics of two predators $X_j(t)$ for $j = 1, 2$ and one prey $S(t)$:
\begin{equation}\label{eq:original_model}
\begin{cases}
\displaystyle \dot{X}_1 = \left(\frac{m_1 S}{S + a_1} - d_1\right) X_1, \\
\displaystyle \dot{X}_2 = \left(\frac{m_2 S}{S + a_2} - d_2\right) X_2, \\
\displaystyle \dot{S} \hspace{0.2cm} = \left(\gamma \left(1-\frac{S}{\kappa}\right) - \frac{m_1}{y_1}\left(\frac{X_1}{S + a_1}\right) - \frac{m_2}{y_2}\left(\frac{X_2}{S + a_2}\right)\right) S, 
\end{cases}
\end{equation}
with initial conditions in $ \mathbb{R}^3_+ \bydef \{(X_1, X_2, S) \, : \, X_1 > 0, \, X_2 > 0, \, S >0\}$.
The system \eqref{eq:original_model} involves ten positive parameters: $\kappa$ is the carrying capacity of the prey, $\gamma$ is its intrinsic rate of increase, and, for the $j$-th predator, $m_j$ is its maximum birth rate, $a_j$ is its half-saturation constant, $d_j$ is its death rate, and $y_j$ is its yield conversion factor.
The functional response in \eqref{eq:original_model} is called Holling’s type II, which is also known as the Michaelis--Menten kinetics in chemistry; see \cite{10.1007/BF00276918, 10.1137/0143066}.
For experimental results on \eqref{eq:original_model}, see \cite{10.1126/science.6767274} and the references therein.

Since $\mathbb{R}_+^3$ is invariant under the dynamics of the system \eqref{eq:original_model}, solutions with initial conditions in $\mathbb{R}_+^3$ remain \emph{positive}, that is, $(X_1(t), X_2(t), S(t)) \in \mathbb{R}^3_+$ for all $t \in \mathbb{R}$.
Biodiversity described by \eqref{eq:original_model} occurs when all species coexist such that $\liminf_{t \to \infty} X_j(t) > 0$ for $j = 1, 2$ and $\liminf_{t \to \infty} S(t) > 0$.
In other words, coexistence is characterized by the survival of both predators and their prey, {thereby disproving the competitive exclusion principle for the system \eqref{eq:original_model}.} 

Numerical results have long suggested that coexistence in $\mathbb{R}^3_+$ may manifest through periodic orbits or more intricate dynamics.
However, only a few results have been analytically proven, and the existence of periodic orbits is based on strong assumptions about the parameters.
We now comment on the relevant literature.
First, the dynamics near both invariant boundary planes
\begin{equation}
Q_1 \bydef \{(X_1, 0, S) \,: \, X_1 >0, \, S > 0\}, \qquad
Q_2 \bydef \{(0, X_2, S) \, : \, X_2 > 0, \, S > 0\}
\end{equation}
are well understood.
Specifically, in $Q_1$, the system \eqref{eq:original_model} has a unique \emph{boundary equilibrium} given by 
\begin{equation}
E_1 \bydef \left(\frac{\gamma y_1 (\kappa \lambda_1 + a_1) (1 - \lambda_1)}{m_1}, 0, \kappa \lambda_1\right), \qquad \text{ where } \qquad   \lambda_1 \bydef \frac{1}{\kappa} \left( \frac{a_1 d_1}{m_1 - d_1}\right),
\end{equation}
%
as we assume $1 - \lambda_1 > 0$ and $m_1 - d_1 > 0$.
Moreover, $E_1$ undergoes a local Hopf bifurcation in $Q_1$ when $2 \kappa \lambda_1 + a_1 - \kappa = 0$, triggering a \emph{boundary limit cycle} $\mathcal{C}_1 \subset Q_1$ for $2 \kappa \lambda_1 + a_1 - \kappa < 0$; see \cite{Sm82}.
Then, positive periodic orbits can bifurcate from $\mathcal{C}_1$ via a local transcritical bifurcation; see \cite{10.1007/BF00305755, 10.1007/BF00276918, Sm82}.
Notice that such periodic solutions are established only in a neighborhood of $Q_1$ in $\mathbb{R}^3_+$, and therefore for a small population size of $X_1$. Since \eqref{eq:original_model} remains unchanged by interchanging the index $j = 1, 2$, we can define $E_2$, $\lambda_2$, and $\mathcal{C}_2$ analogously.
Second, geometric singular perturbation theory is applicable for sufficiently large $\gamma \gg 0$ resulting in a positive periodic orbit; see \cite{10.1016/S0022-0396(02)00076-1}.
Third, perturbing a conserved quantity by considering two small difference assumptions on the parameters, $0 < (a_2 - a_1)/\kappa \ll 1$ and $0 < \lambda_2 - \lambda_1  \ll 1$, yield stable positive periodic orbits far from both boundary planes $Q_1$ and $Q_2$; see \cite{10.1137/0143066}.
Last, with only one small difference assumption, $0 < \lambda_2 - \lambda_1 \ll 1$, a local hybrid Hopf bifurcation occurs by eliminating a line of equilibria, also yielding stable positive periodic orbits far from both boundary planes; see \cite{Jia23}.
We emphasize that all existing results are either inherently local or rely on restricting the choice of parameters.

In contrast, in this article, we present a continuation method to prove a family of positive periodic orbits, which is \emph{global} in the sense that it connects two boundary limit cycles $\mathcal{C}_1$ and $\mathcal{C}_2$.
The family corresponds to a curve in the ten-dimensional parameter space.
Following \cite{HHW1978}, we parameterize such a curve by the carrying capacity $\kappa > 0$.
Moreover, the family consists of \emph{stable} periodic orbits in $\mathbb{R}^3_+$, meaning that the associated \textit{Floquet exponents} of each periodic orbit are $0$, $\mu_1$, and $\mu_2$ such that $\Re(\mu_1) < 0$ and $\Re(\mu_2) < 0$.
Our continuation method does not rely on restricting the choice of parameters, due to the nature of the computer-assisted proof developed in Sections \ref{sec:continuation}--\ref{sec:stability}.
As an application, we choose the following set of parameter values (noticing that $y_1, y_2, \gamma$ can always be rescaled to $1$)
\begin{equation}\label{parameter-set}
a_1 = 10, \quad a_2 = 41, \quad d_1 = 0.8, \quad d_2 = 0.5, \quad m_1 = m_2 = y_1 = y_2 = \gamma = 1,
\end{equation}
such that $a_2 - a_1 = 31$ and $\lambda_2 - \lambda_1 = 1/\kappa$. which yields stable periodic solutions that are not proved in the literature.
This is the content of the next theorem.
Furthermore, intricate dynamics appear to arise along the global families as $a_1$ decreases; see Section \ref{sec:future}.

\begin{theorem}[Global family of stable periodic orbits] \label{thm:main}
For the set of parameter values \eqref{parameter-set},
there exists a global family of stable positive periodic orbits, parameterized by the carrying capacity $\kappa > 0$, which connects both boundary limit cycles; this family lies within a distance $10^{-10}$ (in $C^0$-norm) from the approximation depicted in Figure \ref{fig:tube}.
Specifically, there exist $\hat{\kappa}_1, \hat{\kappa}_2 \in [92, 129]$ such that
\begin{enumerate}
\item[\textit{(i)}] for each $\kappa \in (\hat{\kappa}_1,\hat{\kappa}_2)$, the periodic orbit, denoted by $\mathcal{P}(\kappa)$, is positive and stable;
\item[\textit{(ii)}] $\mathcal{P}(\hat{\kappa}_1) = \mathcal{C}_1$, where $\mathcal{C}_1$ is the boundary limit cycle in $Q_1 \bydef \{(X_1,0,S) \, : \, X_1 > 0, \, S > 0\}$;
\item[\textit{(iii)}] $\mathcal{P}(\hat{\kappa}_2) = \mathcal{C}_2$, where $\mathcal{C}_2$ is the boundary limit cycle in $Q_2\bydef\{(0,X_2,S) \, : \, X_2 > 0, \, S > 0\}$.
\end{enumerate}
\end{theorem}

\begin{figure}
\centering
\begin{subfigure}[b]{0.48\textwidth}
\includegraphics[width=\textwidth]{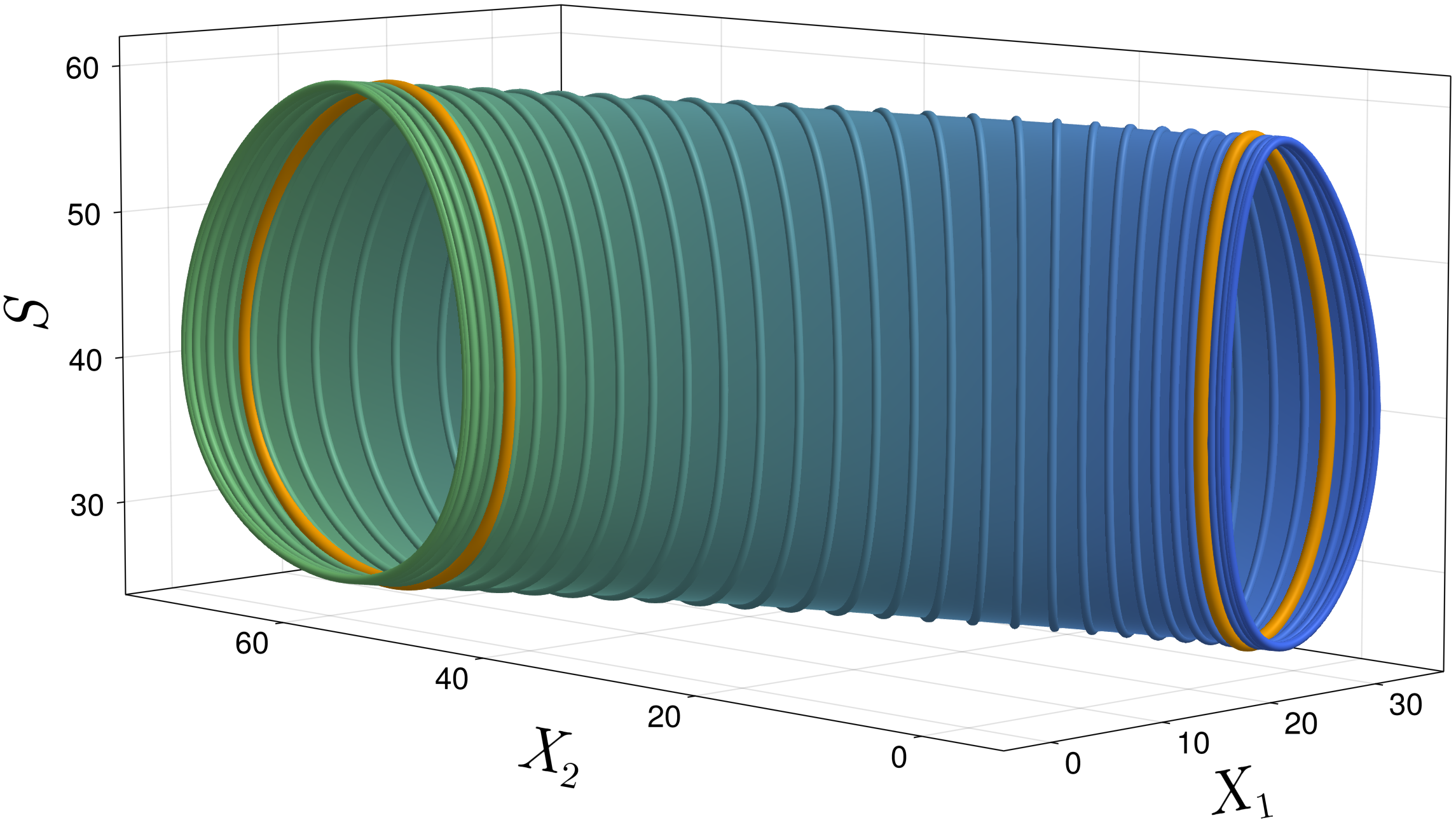}
\caption{Global family of stable positive periodic orbits for $\kappa \in [92, 129]$.
}
\end{subfigure}
\hspace{0.04\textwidth}
\begin{subfigure}[b]{0.44\textwidth}
\includegraphics[width=\textwidth]{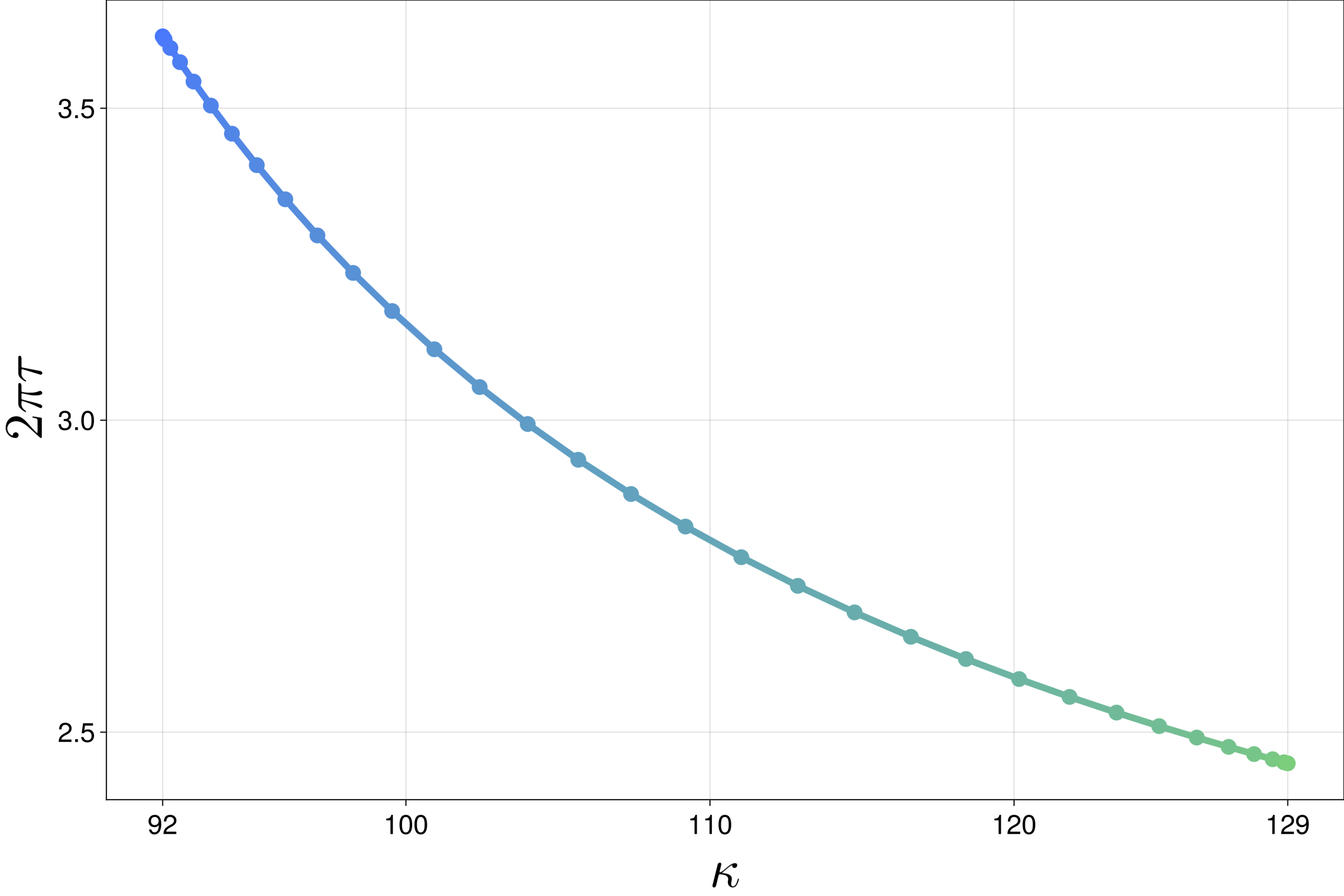}
\caption{Minimal period $2\pi\tau$.}
\end{subfigure}
\caption{The parameter values are set to \eqref{parameter-set}.
(a) Fourier--Chebyshev approximation (with $K = 20$, $N = 30$, see Section \ref{sec:contraction}) of the global family of stable positive periodic orbits to the system \eqref{eq:original_model} obtained in Theorem \ref{thm:main}. Within a distance $10^{-10}$ (in $C^0$-norm), there exists an exact family of stable periodic orbits. The true family is so close to the approximation that they are visually indistinguishable. The orange rings corresponds to the boundary limit circles $\mathcal{C}_1$ and $\mathcal{C}_2$ at $\hat{\kappa}_1 \approx 93.0545$ and $\hat{\kappa}_2 \approx 126.3145$, respectively. (b) Dependence of the minimal period on $\kappa \in [92, 129]$ along the global family. Here $\tau > 0$ is a time-rescaling parameter introduced in \eqref{eq:normalization}.}
\label{fig:tube}
\end{figure}

Theorem \ref{thm:main} is an affirmative answer to the \emph{stable connection problem} proposed by Butler and Waltman \cite[page 309]{10.1007/BF00276918}:
\begin{center}
\em
\noindent
The connection of these two [boundary] limit cycles and \\the determination of their stability were not established.
\end{center}
Numerical evidence (see \cite{HHW1978} for instance) hints at a broad parameter region that supports stable positive periodic orbits.
However, aside from largely perturbative or asymptotic results, there has not been much progress in solving the problem,{ except for \cite{10.1137/0143066}, in which the results rely on restricting the small difference assumptions on the parameters $0 < (a_2 - a_1)/\kappa \ll 1$ and $0 < \lambda_2 - \lambda_1  \ll 1$.}

We highlight the following three aspects concerning the novelty and methodology of the proof of Theorem \ref{thm:main}:
\begin{itemize}
\item We prove a global family of periodic solutions without imposing any small difference assumptions on the parameters.
Hence, we can obtain periodic orbits with large amplitude, and, moreover, our method is applicable to much broader parameter regions than those considered in the relevant literature \cite{10.1007/BF00305755, 10.1007/BF00276918, 10.1137/0143066, 10.1016/S0022-0396(02)00076-1, Jia23, Sm82}.

\item We determine the stability for all periodic orbits of a global family.
Notably, the proof of stability follows a similar strategy to the one used for the existence.
In both cases, we formulate a zero-finding problem, where we exploit the local contraction of a Newton-type fixed-point operator, centered around a high-order approximation of the solution.

\item The techniques used to prove Theorem \ref{thm:main} readily extend to other vector fields such as Holling's type III \cite{10.1007/s10884-008-9102-9}, the Beddington--DeAngelis type \cite{HsuRuanYang13}, an external inhibitor \cite{DLS}, and other functional responses \cite{AbramsHolt03, AbramsHolt02, HM, Keener85}.
\end{itemize}

We prove Theorem \ref{thm:main} using a rigorous continuation method, recasted as a single zero-finding problem $F(\chi) = 0$.
This map $F$ incorporates an auxiliary system to \eqref{eq:original_model} whose periodic orbits coincide, allowing us to continue through the local transcritical bifurcations from the boundary limit cycles.
We employ a spectral method, so that periodic solutions are expressed as Fourier series, where the dependency on the continuation parameter $\kappa > 0$ is expressed by expanding the Fourier coefficients as an infinite series of Chebyshev polynomials of the first kind.
Then, the existence of the zero of $F$, thereby of the family of periodic orbits, is proved by a local contraction argument.
The set of sufficient conditions for the appropriate fixed-point operator to be a local contraction amounts to satisfying inequalities relating norm estimates, expressed in terms of a numerical approximation of the branch of periodic orbits.
The computer is employed to perform large number needed to verify the contraction hypotheses of calculations together with interval arithmetic to prevent rounding errors.
While the strategy detailed here is aimed at addressing the connection problem, an other application was also made in celestial mechanics by one of the author of this article to resolve Marchal's conjecture in the three-body problem \cite{Marchal}.
The proof of stability is done after the existence of the family, and uses a similar contraction strategy, it also utilizes the analytic regularity of the solution curve inferred from the functional framework.

Computer-assisted proofs (abbr. CAPs) based on a posteriori validation go back to the mid-1980s with the works on the Feigenbaum conjectures \cite{feigenbaum2, feigenbaum3, feigenbaum1}.
The method followed in this article is inspired by a long history of techniques, such as \cite{MR0151678, MR0173839,DAY_ValidatedContinuation, MR2394226, MR3454370, MR0231218, MR1639986}.
A notable result of application in the field of dynamical systems, is the proof of the chaotic attractor in the Lorenz system \cite{Lorenz}; the computer was involved in both the discovery and, ultimately, the proof of its existence \cite{Tucker}.
For a more thorough account on the development of CAPs in the recent years, the interested readers are referred to the survey papers \cite{GOMEZ_PDESurvey, KOCH_ComputeAssisted, NAKAO_VerifiedPDE, VANDENBERG_Dynamics}. See also \cite{Lucia} for an introductory account on proving the existence and continuation of periodic orbits using CAPs.

The proof of Theorem \ref{thm:main} uses several new techniques in CAPs (see also Table~\ref{tab:related-work}), notably:

\begin{itemize}
\item[(1)] We use a \emph{blow-up} approach to isolate zeros; see also \cite{Marchal, 10.1007/s10884-023-10279-x, 10.1137/20M1343464}.
This allows us to desingularize the continuation as the family undergoes a transcritical bifurcation on the boundary planes $Q_1$ and $Q_2$.

\item[(2)] We employ a Chebyshev continuation procedure (see \cite{RigContPaper}) to prove the global family of periodic orbits, and auxiliary data (periods and amplitudes), as a single isolated zero of a map.
See \cite{AriGazKoc21,RigCont} for related works using high-order interpolation of solution curves.
This is in contrast to the more classical approach of using a piecewise-linear approximation of the solution curve, requiring several uniform contractions; see, e.g., \cite{Arioli2006,BouwevandenBerg2021,BreLesVan13,BerLesMis10}.

\item[(3)] Our choice of norm allows us to prove the analytic dependence of the family on the continuation parameter, and to control the derivatives with respect to the parameter.

\item[(4)] Our proof is efficient in the sense that we handle non-polynomial nonlinearities directly (in the vein of \cite{Payan}), rather than using a so-called polynomial embedding \cite{Henot2021-cm}, which results in studying vector fields in higher dimensions.
\end{itemize}

\begin{table}[h]
\centering
\caption{Non-exhaustive comparison of CAP strategies.}
\label{tab:related-work}
\renewcommand{\arraystretch}{1.1}
\begin{tabular}{p{0.19\linewidth} p{0.36\linewidth} p{0.36\linewidth}}
\toprule
\textbf{Theme} 
& \textbf{Alternative CAP strategies} 
& \textbf{Present approach} \\
\midrule

Bifurcations
& Blow-up for Hopf bifurcations \cite{10.1007/s10884-023-10279-x,10.1137/20M1343464}
& Blow-up for transcritical bifurcations of periodic orbits from invariant subsets \\
\midrule

Global continuation
& Piecewise-linear approximation of family needing multiple uniform contractions \cite{Arioli2006,BouwevandenBerg2021,BreLesVan13,BerLesMis10}, or Taylor expansion of the family \cite{AriGazKoc21}
& Chebyshev expansion of the family \\
\midrule

Nonlinearities
& Polynomial embedding \cite{Henot2021-cm}
& Direct treatment \\
\bottomrule
\end{tabular}
\end{table}

The code for the proof is implemented in Julia \cite{Julia} and is found at \cite{Code}.
The rigorous computations use the packages \texttt{RadiiPolynomial} \cite{RadiiPolynomial.jl} and \texttt{IntervalArithmetic} \cite{IntervalArithmetic.jl}.
All figures are generated using \texttt{Makie} \cite{Makie.jl}.

This article is structured as follows. { In Section \ref{sec:bioapplications}, we present a general class of ecological systems to investigate the competitive exclusion principle and highlight the effectiveness of our CAP method. In Section \ref{sec:general_framework}, we provide a general framework to desgingularize transcritical bifurcations.}
In Section \ref{sec:continuation}, we present the rigorous continuation method, which proves the existence part of Theorem \ref{thm:main}.
In Section \ref{sec:stability}, we address the stability of periodic orbits, where some details are deferred to Appendix \ref{app:contraction}.
Finally, in Section \ref{sec:future}, we outline potential future work.

{
\subsection{Further Applications: the competitive exclusion principle}\label{sec:bioapplications}

{A central paradigm in ecology is the \emph{competitive exclusion principle} \cite{gause1934, MathBioBook}, which states that different species competing for the exact same limited resource cannot stably coexist. In other words, if their ecological niches are identical, one species will eventually exclude the others. While this principle provides a useful heuristic in ecology, there is also considerable mathematical interest in rigorously establishing conditions under which it holds or fails. Notably, H.L. Smith’s work on \emph{monotone dynamical systems} \cite{smith1995monotone} characterizes settings where competitive interactions enforce exclusion, and where cooperative or more complex structures may allow coexistence. In more applied domains, explicit examples where the principle fails are considered significant contributions in their own right (see e.g., \cite{SakrefAutoExc24}), which demonstrates failure in the context of certain autocatalytic networks. One natural way to demonstrate the failure of the competitive exclusion principle is by establishing the existence of positive periodic orbits, which guarantee ongoing coexistence. In this context, our CAP framework provides a robust method to establish global continuation of such periodic orbits, beyond perturbation arguments. The framework itself opens a principled path toward identifying and rigorously verifying instances of coexistence.
\\
\indent To exemplify, we present the \textit{problem of resource competition} in Lotka--Volterra systems; see the pioneering work \cite{LeTu75}.} Suppose that there are $M$ competing predator species consuming $N$ prey resources. The problem requires to determine all pairs $(M,N) \in \mathbb{N}^2$ with $M \ge N$ such that the competitive exclusion principle fails, thus supporting biodiversity. Specifically, let $X_j(t)$ for $j = 1,\dots, M$ and $S_k(t)$ for $k = 1,\dots, N$ denote the population density of the $j$-th predator and $k$-th prey, respectively.
We consider the following general system that describes the population dynamics: 
\begin{equation}\label{eq:general_model}
\begin{cases}
\displaystyle \dot{X}_j = \left(g_j(S_1, \dots, S_N) - d_j\right) X_j, & j = 1,\dots, M, \\
\displaystyle \dot{S}_k \hspace{0.2cm} = h_k(S_k) - \sum_{\ell = 1}^M c_{k \ell} \, g_\ell(S_1, \dots, S_N) \, X_\ell, & k = 1,\dots, N. 
\end{cases}
\end{equation}
Here, for the $j$-th predator, $g_j : \mathbb{R}^N_+ \to \mathbb{R}_+$ is the conversion rate to its biomass and $d_j > 0$ is its death rate. For the $k$-th prey, $h_k : \mathbb{R}_+ \to \mathbb{R}_+$ is the growth rate when the predation is absent, and $c_{k \ell} \in (0,1)$, which satisfies $\sum_{k = 1}^N c_{k j} = 1$ for each $j = 1,\dots,M$, is the fraction of the $\ell$-th predator consuming it.

To choose $g_j(S_1, \dots, S_N)$ and $h_k(S_k)$, the literature \cite{BaLiSm06, HuWe99, HuWe02, LiSm01, LiSm03} assumes the so-called law of the minimum
\begin{equation}
g_j(S_1, \dots, S_N) = \min_{i = 1,\dots,N} \{ f_{ij}(S_i)\},
\end{equation}
where $f_{ij}(S_i)$ is of Holling's type II.
Note that the system \eqref{eq:original_model} corresponds to the special case $(M, N) = (2,1)$. The problem of resource competition is uncharted for the \textit{logistic quadratic growth}
\begin{equation} \label{eq:logistic}
h_k(S_k) =  \gamma (S_k^0 - S_k) S_k,
\end{equation}
where $\gamma > 0$ is the intrinsic rate of increase and $S_k^{0} > 0$ is the growth limiting constant.
For the \textit{chemostat affine growth}
\begin{equation} \label{eq:chemostat}
h_k(S_k) = \gamma (S_k^{0} - S_k),
\end{equation}
partial solutions are available.
For the cases $(M,N) = (2,2), (3, 2)$, the competitive exclusion principle holds; see \cite{LiSm01}.
For the case $(M,N) = (3,3)$, several positive periodic orbits exist; see \cite{BaLiSm06}.
However, it remains open whether periodic orbits exist for other cases.
Still, numerical evidence suggests that positive periodic orbits exist for the cases $(M, N) = (4,4), (5,5)$; see \cite{HuWe01}. 

The framework of our computer-assisted techniques (see Section \ref{sec:general_framework}) is largely model independent, and thus it paves a way to solve the problem of resource competition, for both growth rates \eqref{eq:logistic}--\eqref{eq:chemostat}. Indeed, for each fixed pair $(M,N)$ whose numerical evidence exhibits periodicity, we can employ our techniques not only to prove the existence of positive periodic orbits and thus solve the problem, but also their stability property and continuations.
}



\subsection{General framework to desingularize transcritical bifurcations} \label{sec:general_framework}

The stable connection problem involves a continuation of positive periodic orbits through both boundary planes $Q_1$ and $Q_2$, where a transcritical bifurcation occurs.
Since the validation of the continuation hinges on the local isolation of the solution, we establish a \emph{blow-up method} that yields an auxiliary system of equations for which the undesired branch of periodic solutions is removed.
In this section, we present this approach and explain why this strategy achieves such local isolation in the most general setting.

Consider the general system
\begin{equation}\label{eq:general_odes}
\partial_t x = f(x; \mu),
\end{equation}
where $f = (f_1, \dots, f_n)$ is analytic, $x = (x_1, \dots, x_n) \in \mathcal{X} \subset \mathbb{R}^n$, and $\mu \in U \subset \mathbb{R}$ is a parameter.
We are interested in the following situation:
\begin{enumerate}
\item[{\bf (A1)}] There exists an invariant subset $Q \bydef \{x \in \mathcal{X} \, : \, G(x; \mu) = 0\}$ of $\mathbb{R}^n$, where $G : \mathcal{X} \times U \to \mathbb{R}$ is analytic such that $\partial_{x_j} G(x; \mu) \ne 0$ for some $ j = 1, \dots, n$ and all $(x; \mu) \in \mathcal{X} \times U$, i.e., $\partial_{x_j} G(x; \mu)$ is invertible.
Without loss of generality, we assume $j = 1$.
\item[{\bf (A2)}] There exist two families $\mathcal{P}_1(\mu)$ and $\mathcal{P}_2(\mu)$ of periodic orbits parametrized by $\mu \in U$ such that $\mathcal{P}_1(\mu) \subset Q$ for all $\mu \in U$ and $\mathcal{P}_2(\mu) \cap Q = \emptyset$ for $\mu \ne \mu_{\mathrm{bif}} \in U$. Moreover, $\mathcal{P}_1(\mu_\mathrm{bif}) = \mathcal{P}_2(\mu_\mathrm{bif})$, i.e., the two families meet on the invariant subset $Q$ at $\mu = \mu_{\mathrm{bif}}$.
\end{enumerate}

Let $x = x(t; \mu)$ be a solution of~\eqref{eq:general_odes}.
We introduce the \emph{blow-up coordinates} $\zeta : U \to \mathbb{R}$ and $u : \mathbb{R} \times U \to \mathbb{R}$ such that
\begin{equation}\label{eq:blow_up_coord}
G(x (t; \mu); \mu) = \zeta(\mu) u(t; \mu).
\end{equation}
By {\bf (A1)}, the implicit function theorem implies that $x_1 = \omega(\zeta(\mu) u, x_2 , \dots, x_n ; \mu)$ for some smooth function $\omega$.
Differentiating~\eqref{eq:blow_up_coord} with respect to $t$ yields
\begin{equation}\label{eq:relation}
\zeta(\mu) \partial_t u(t; \mu) 
= \nabla G(z(t; \mu); \mu) f(z(t; \mu); \mu), \qquad z \bydef (\omega (\zeta(\mu) u, x_2 , \dots, x_n ; \mu) , x_2 , \dots, x_n ),
\end{equation}
where $\nabla$ is the gradient with respect to the variables $x = (x_1, \dots, x_n)$.
Also, if $t \mapsto x(t; \mu)$ represents the periodic solution to~\eqref{eq:general_odes} associated to $\mathcal{P}_2$, then by {\bf (A2)}, we have $\mathcal{P}_2(\mu_\mathrm{bif}) \subset Q$ and so $\zeta(\mu_\mathrm{bif}) = 0$.
From the invariance of $Q = \{ x \in \mathcal{X} \, : \, G(x; \mu) = 0\}$, we get
\begin{equation}
\nabla G(x; \mu_\mathrm{bif}) f(\omega(0, x_2 , \dots, x_n ; \mu_\mathrm{bif}), x_2 , \dots, x_n ; \mu_\mathrm{bif}) = 0.
\end{equation}
This guarantees that $\lim_{\mu \to 0} \partial_t u(t; \mu)$ exists, as seen by expanding the right-hand side of~\eqref{eq:relation} near $x_1 = 0$, such that
\begin{align} 
\begin{split}
\label{ratio-relation}
\partial_t u(t; \mu) & = \frac{1}{\zeta(\mu)} \nabla G( z(t ;\mu); \mu ) f(z(t;\mu) ;\mu)
\\
& = \sum_{\ell = 1}^\infty c_\ell (x_2 (t;\mu), \dots, x_n (t;\mu); \mu) \zeta(\mu)^{\ell-1} u(t;\mu)^{\ell},
\end{split}
\end{align}
where
\begin{equation}
c_\ell (x_2 , \dots, x_n ; \mu) \bydef \left.\partial_{x_1}^\ell \right|_{x_1 = 0} \nabla G( z; \mu ) f(z ;\mu), \qquad \ell \ge 1.
\end{equation}
This yields the \emph{auxiliary system}
\begin{equation}\label{eq:auxiliary_blowup}
\begin{cases}
\displaystyle\partial_t \begin{pmatrix}
u \\ x_2  \\ \vdots \\ x_n 
\end{pmatrix} =
\begin{pmatrix}
\frac{1}{\zeta(\mu)} \nabla G( \omega, x_2 , \dots, x_n; \mu ) f(\omega, x_2 , \dots, x_n ;\mu)
\\
f_2 ( \omega, x_2 , \dots, x_n ; \mu) \\
\vdots \\
f_n ( \omega, x_2 , \dots, x_n ; \mu)
\end{pmatrix}, \\
G(\omega, x_2 , \dots, x_n ; \mu) = \zeta(\mu) u, \\
u(0; \mu) = 1.
\end{cases}
\end{equation}
We have that if $(u,x_2,\dots,x_n)$ is a periodic solution, for some $\zeta$, to the auxiliary system~\eqref{eq:auxiliary_blowup}, then $(\zeta u, x_2, \dots, x_n)$ is a periodic solution to the original system~\eqref{eq:general_odes}.
Provided that $\zeta \ne 0$, the periodic solution does not lie in $Q$.
Importantly, periodic orbits are locally isolated whenever $\zeta = 0$.
For completeness, we state this in the following lemma. However, we stress that verifying its hypotheses is unnecessary for our proof; it serves to justify why using the coordinates \eqref{eq:blow_up_coord} and studying the auxiliary system \eqref{eq:auxiliary_blowup} is compatible with a contraction argument (detailed in Section~\ref{sec:contraction}) to continue a branch of periodic solutions through a non-degenerate transcritical bifurcation.

\begin{proposition}[Local isolation of periodic orbits] \label{prop-blow-up}
Under the assumptions {\bf (A1)} and {\bf (A2)}, the family of periodic orbits $\mathcal{P}_2(\mu)$, seen in the coordinates \eqref{eq:blow_up_coord} and solving \eqref{eq:auxiliary_blowup}, is locally isolated near $\mu =0$.
\end{proposition}

\begin{proof}
Let $x_*(t; \mu)$ denote the periodic solution associated with $\mathcal{P}_2(\mu)$; we need to prove local uniqueness at $\mu=\mu_\mathrm{bif}$.
Without loss of generality, we assume $\mu_\mathrm{bif} = 0$.

At $\mu = 0$, since $Q = \{x \in \mathcal{X} \, : \, G(x; \mu) = 0\}$ is an invariant subset by {\bf (A1)}, each $x_j$-equation of \eqref{eq:auxiliary_blowup}, for $j = 2,\dots, n$, is decoupled from the $u$-equation.
From {\bf (A2)}, we know that $x_*(t; 0)$ is contained in $Q$ and solves the decoupled system (i.e., without the $u$-coordinate).
Moreover, due to this decoupling, the first equation becomes a linear non-autonomous ODE with periodic coefficients
\[
\partial_t u(t;0) = \phi(t) u(t;0).
\]
The general solution reads $u (t; 0) = C e^{\int_0^t \phi(s) \, \mathrm{d}s} > 0$, where $C = 1/e$ is uniquely determined by the scaling $u(0;0)=1$.
\end{proof}

In the specific predator-prey system \eqref{eq:general_model}, $\mathcal{X}$ is the closure of $\mathbb{R}_+^3$, $\mu \in U = (0, \infty)$ is the carrying capacity $\kappa$, $Q = \{x \in \mathcal{X} \, : \, x_1 = 0\}$ (resp., $Q = \{x \in \mathcal{X} \, : \, x_2 = 0\}$) is the boundary plane $Q_2$ (resp., $Q_1$), $\mathcal{P}_1(\mu)$ is the family of boundary limit cycles, and $\mathcal{P}_2(\mu)$ is the continuation of positive periodic orbits.
Also, the coordinates \eqref{eq:blow_up_coord} reduce to $X_1 = a_1 u_1$ and $X_2 = a_2 u_2$ to handle, respectively, the two boundary planes $Q_1$ and $Q_2$. Moreover, $\omega$ is known explicitly to be $\omega = au$.



\section{Proof of analytic families of periodic orbits}
\label{sec:continuation}

We normalize the system \eqref{eq:original_model} as in \cite{10.1137/0143066}:
\begin{equation}\label{eq:normalization}
\begin{aligned}
&x_j(t) \bydef \frac{1}{\kappa} \frac{m_j}{\gamma y_j} X_j \left(\frac{1}{\gamma} \tau t\right), \qquad
s(t) \bydef \frac{1}{\kappa} S\left(\frac{1}{\gamma} \tau t\right), \\
&\alpha_j \bydef \frac{1}{\kappa} a_j, \qquad
\lambda_j \bydef \frac{1}{\kappa} \left(\frac{a_j d_j}{m_j - d_j}\right), \qquad
\delta_j \bydef \frac{m_j - d_j}{\gamma},
\end{aligned}
\end{equation}
where $\tau > 0$ is a parameter introduced to scale the period to $2\pi$, and which will be determined later.
Then, we obtain the following \emph{rescaled system}:
\begin{equation}\label{eq:model}
\begin{cases}
\displaystyle \dot{x}_j = \tau \delta_j \left(\frac{s - \lambda_j}{s + \alpha_j}\right) x_j, & j = 1, 2,\\
\displaystyle \dot{s} \hspace{0.2cm} = \tau \left(1 - s - \frac{x_1}{s + \alpha_1} - \frac{x_2}{s + \alpha_2}\right) s.
\end{cases}
\end{equation}
Our objective is to find a one-parameter family of positive periodic solutions to \eqref{eq:model} connecting the boundary planes $Q_1 = \{(x_1, 0, s) \, : \, x_1 > 0, \, s > 0\}$ and $Q_2 = \{(0, x_2, s) \, : \, x_2 > 0, \, s>0\}$ by increasing the parameter $\kappa$ from $\kappa_1$ to $\kappa_2$.
Although the parameter $\kappa$ no longer appears explicitly in \eqref{eq:model}, its variation is captured by the parameters $\alpha_j$ and $\lambda_j$.
The values $\kappa_1$ and $\kappa_2$ will be deduced from the numerics, and we will verify a posteriori that indeed the family crosses $Q_1$ and $Q_2$.
Namely, the branch of periodic orbits at $\kappa_1$ and $\kappa_2$ will reside outside $\mathbb{R}_+^3$, and thus we introduce the notation $\hat{\kappa}_1$ and $\hat{\kappa}_2$, as reported in Theorem \ref{thm:main}, to emphasize the parameter values for which a local transcritical bifurcation occurs at $Q_1$ and $Q_2$, respectively.

Since only the inverse of $\kappa$ appears in \eqref{eq:normalization}, given $0 < \kappa_1 \le \kappa_2$, we parameterize $\kappa$ (and, more directly, $\alpha_j$ and $\lambda_j$) by
\begin{equation}\label{eq:kappa}
\eta \in [-1, 1] \quad \mapsto \quad \kappa(\eta) \bydef \frac{2 \kappa_1 \kappa_2}{\kappa_1 + \kappa_2 + (\kappa_1 - \kappa_2)\eta}
\end{equation}
such that $\kappa(-1) = \kappa_1$ and $\kappa(1) = \kappa_2$.
The variable $\eta$ now plays the role of the continuation parameter, and we write
\begin{equation}\label{eq:variables}
\alpha_j (\eta) = \frac{1}{\kappa(\eta)} a_j, \qquad
\lambda_j (\eta) = \frac{1}{\kappa(\eta)} \left( \frac{a_j d_j}{m_j - d_j} \right),
\end{equation}
to emphasize their dependency. 

{
The roadmap for proving the existence of the family is the following:
\begin{itemize}
\item Formulation: we set up a zero-finding problem whose zeros are in one-to-one correspondence with families of periodic solutions to the rescaled system \eqref{eq:model}.
\item Analysis: we derive a set of sufficient conditions for a Newton type fixed-point operator to be a contraction near a high-order numerical approximation of the family.
\item Computations: we perform the desired computations to verify the contraction conditions; see the code available in~\cite{Code}.
The code incorporates some relevant checks discussed in Section~\ref{sec:aposteriori} to ensure that the proven family is real-valued and indeed crosses the boundary planes $Q_1$ and $Q_2$.
It also contains the proof of the stability, which is based on a contraction argument as well, though for a different set of equations to solve.
\end{itemize}
}

\subsection{Zero-finding problem}

The parameter continuation of positive periodic orbits encounters a singularity at the boundary planes $Q_1$ and $Q_2$.
In each plane, there is a periodic orbit that persists for a broad range of parameters, and as we continue along the parameter curve, the two distinct families of periodic orbits -- one in $\mathbb{R}_+^3$ and the other in $Q_j$ -- intersect.
Such an intersection obstructs the continuation as it violates the implicit function theorem.
While it is in principle possible that more degenerate cases occur, we only consider the standard scenario where the interior family connects to a boundary limit cycle in $Q_1$ and $Q_2$ via the transcritical bifurcation studied in \cite{10.1007/BF00305755, 10.1007/BF00276918, Sm82}.

{Therefore, this falls in the scope of Section \ref{sec:general_framework}, and} we consider $u = (u_1, u_2, u_3)$ with
\begin{equation}
\begin{cases}
x_j (t, \eta)  = \zeta_j (\eta) u_j (t, \eta), & j = 1, 2, \\
s(t, \eta) \hspace{0.2cm} = u_3(t, \eta).
\end{cases}
\end{equation}
This leads to the following auxiliary system for the continuation parameter $\eta \in [-1, 1]$:
\begin{equation}\label{eq:model_blowup}
\begin{cases}
\displaystyle \partial_t u_j = \tau \delta_j \left(\frac{u_3 - \lambda_j(\eta)}{u_3 + \alpha_j(\eta)}\right) u_j, & j = 1, 2,\\
\displaystyle \partial_t u_3 = \tau \left(1 - u_3 - \zeta_1 \frac{u_1}{u_3 + \alpha_1(\eta)} - \zeta_2 \frac{u_2}{u_3 + \alpha_2(\eta)}\right) u_3, \\
u_j (0, \eta) = 1, & j =1, 2.
\end{cases}
\end{equation}
Then, a family of periodic orbits to the auxiliary system \eqref{eq:model_blowup}, parameterized by $\eta \in [-1, 1]$, is a zero of the map $F$ defined (at this stage only formally), for $\chi = (\tau, \zeta_1, \zeta_2, u)$, by
\begin{equation}\label{eq:zero_finding_problem}
F(\chi) \bydef
\begin{pmatrix}
\rho(u) \\
\partial_t u - \tau f(\zeta_1, \zeta_2, u)
\end{pmatrix},
\end{equation}
where, $u = (u_1, u_2, u_3)$,
\begin{align}
\rho(u) &\bydef
\begin{pmatrix}
\displaystyle\int_0^{2\pi} \langle u(t, \,\cdot), \partial_t \Gamma(t, \,\cdot) \rangle \, \mathrm{d}t \label{eq:rho}\\
\ell(u_1) - 1 \\
\ell(u_2) - 1
\end{pmatrix}, \\
f(\zeta_1, \zeta_2, u) &\bydef
\begin{pmatrix}
\displaystyle\delta_1 \left(\frac{u_3 - \lambda_1}{u_3 + \alpha_1}\right) u_1 \\
\displaystyle\delta_2 \left(\frac{u_3 - \lambda_2}{u_3 + \alpha_2}\right) u_2 \\
\displaystyle\left(1 - u_3 - \zeta_1 \frac{u_1}{u_3 + \alpha_1} - \zeta_2 \frac{u_2}{u_3 + \alpha_2}\right) u_3
\end{pmatrix},
\end{align}
{for some map $\ell$ mimicking the normalization constraints $u_1(0, \eta) = u_2(0, \eta) = 0$, and given explicitly at the end of this section.}
We emphasize again that $\alpha_j = \alpha_j (\eta), \lambda_j = \lambda_j (\eta)$ are parameterized by $\eta$ as described in \eqref{eq:variables}.
The constraint
\begin{equation}\label{eq:phase_condition}
\int_0^{2\pi} \langle u(t, \eta), \partial_t \Gamma(t, \eta) \rangle \, \mathrm{d}t = 0,
\end{equation}
where $\langle \cdot, \cdot \rangle$ denotes the standard inner product on $\mathbb{C}^3$,
fixes the phase of the periodic solutions to remove their time-translation invariance. 
Here, $\Gamma$ denotes some periodic function near the one we seek; its time derivative $\partial_t \Gamma$ may simply be estimated numerically {and described by a finite number of Fourier--Chebyshev coefficients as specified at the end of this section}.

We now detail the functional analytic framework used to define $F$.
The approach is spectral.
We denote the $n$-th \emph{Chebyshev coefficient} of a function $\psi \in C^\infty([-1,1],\mathbb{C})$, namely
\begin{equation}
\psi_n = \frac{1}{2\pi} \int_{-1}^1 \frac{\psi (\eta) T_n (\eta)}{\sqrt{1 - \eta^2}} \, \mathrm{d}\eta, \qquad n \in \mathbb{N}_0 \bydef \mathbb{N} \cup \{0\}.
\end{equation}
Here $T_n : [-1, 1] \mapsto [-1, 1]$ represents the $n$-th Chebyshev polynomial of the first kind, satisfying the identity
\begin{equation}\label{eq:cos_cheb}
T_n(\cos(\theta)) = \cos(n\theta).
\end{equation}

The $k$-th \emph{Fourier coefficient} of a function $\phi \in C^\infty (\mathbb{R}/2\pi\mathbb{Z} \times [-1,1], \mathbb{C})$ is denoted by a subscript $k$ as follows:
\begin{equation}
\phi_k (\eta) = \frac{1}{2\pi} \int_0^{2\pi} \phi (t, \eta) e^{i k t} \, \mathrm{d}t, \qquad \eta \in [-1, 1], \quad k \in \mathbb{Z},
\end{equation}
so that $\phi_k \in C^\infty ([-1, 1], \mathbb{C})$ for all $k \in \mathbb{Z}$.
The \emph{Fourier--Chebyshev coefficients} are denoted by the double subscript $n,k$ as follows
\begin{equation}
\phi_{n,k} = \frac{1}{4\pi^2} \int_0^{2\pi} \int_{-1}^1 \frac{\phi (t, \eta) e^{i k t} T_n (\eta)}{\sqrt{1 - \eta^2}} \, \mathrm{d}\eta \, \mathrm{d}t, \qquad n \in \mathbb{N}_0, \quad k \in \mathbb{Z}.
\end{equation}
For $\nu \ge 1$, we introduce
\begin{align}
\mathcal{P}_\nu &\bydef \left\{ \psi \in C^\infty ([-1, 1], \mathbb{C}) \, : \, \| \psi \|_{\mathcal{P}_\nu} \bydef \sum_{n \in \mathbb{Z}} |\psi_{|n|}| \nu^{|n|} < \infty \right\}, \\
\mathcal{W}_\nu &\bydef \left\{ \phi \in C^\infty (\mathbb{R}/2\pi\mathbb{Z} \times [-1, 1], \mathbb{C}) \, : \, \| \phi \|_\nu \bydef \sum_{k \in \mathbb{Z}} \|\phi_k\|_{\mathcal{P}_\nu} = \sum_{n, k \in \mathbb{Z}} |\phi_{|n|,k}| \nu^{|n|} < \infty \right\}.
\end{align}
The choice of $\nu \ge 1$ is related to the regularity of the function $\psi$;
{ see \cite{Tre13} for an extensive discussion.
For our purposes, we note that, whenever} $\nu > 1$, term-by-term differentiation of the series is well defined up to any order, which will play an important role to verify that our family of periodic orbits crosses each of the boundary planes $Q_1$ and $Q_2$ exactly once; see Section \ref{sec:aposteriori}.
Unlike Taylor series, even non-analytic functions may be represented by Chebyshev series, and in the analytic case, we always have convergence of the series in a Bernstein ellipse.

{
Furthermore, $\mathcal{P}_\nu$ forms a Banach algebra with respect to the multiplication of functions, specifically
\begin{equation}\label{eq:banach_algebra1}
\| \phi \psi \|_{\mathcal{P}_\nu} \le \| \phi\|_{\mathcal{P}_\nu} \| \psi\|_{\mathcal{P}_\nu}, \qquad \text{for all } \phi, \psi \in \mathcal{P}_\nu,
\end{equation}
where, for all $\phi, \psi \in \mathcal{P}_\nu$, we have
\begin{equation}\label{eq:conv1}
( \phi \psi )_n = \sum_{n' \in \mathbb{Z}} \phi_{|n-n'|} \psi_{|n'|}, \qquad n \in \mathbb{N}_0.
\end{equation}
Similarly, $\mathcal{W}_\nu$ forms a Banach algebra
\begin{equation}\label{eq:banach_algebra2}
\| \phi \psi \|_\nu \le \| \phi\|_\nu \| \psi\|_\nu, \qquad \text{for all } \phi, \psi \in \mathcal{W}_\nu,
\end{equation}
where, for all $\phi, \psi \in \mathcal{W}_\nu$, we have
\begin{equation}\label{eq:conv2}
( \phi \psi )_{n, k} = \sum_{n', k' \in \mathbb{Z}} \phi_{|n-n'|, k - k'} \psi_{|n'|, k'}, \qquad n \in \mathbb{N}_0, \quad k \in \mathbb{Z}.
\end{equation}
}

Consider the Banach space
\begin{equation}
\mathcal{X}_\nu \bydef \mathcal{P}_\nu \times \mathcal{P}_\nu \times \mathcal{P}_\nu \times \mathcal{U}_\nu,
\end{equation}
endowed with the norm
\begin{equation}
\| \chi \|_{\mathcal{X}_\nu} \bydef \| \tau \|_{\mathcal{P}_\nu} + \| \zeta_1 \|_{\mathcal{P}_\nu} + \| \zeta_2 \|_{\mathcal{P}_\nu} + \| u \|_{\mathcal{U}_\nu}, \qquad \text{for all } \chi = (\tau, \zeta_1, \zeta_2, u) \in \mathcal{X}_\nu,
\end{equation}
where 
\begin{equation}
\mathcal{U}_\nu \bydef \left\{ u = (u_1, u_2, u_3) \in \mathcal{W}_\nu^3 \, : \, \| u \|_{\mathcal{U}_\nu} \bydef \sum_{j=1}^3 \| u_j \|_\nu < \infty \right\}.
\end{equation}
Then, we have $F : \mathcal{D}(F) \subset \mathcal{X}_\nu \to \mathcal{X}_\nu$ with
\begin{equation}\label{eq:domain}
\mathcal{D}(F) = \left\{ \chi = (\tau, \zeta_1, \zeta_2, u) \in \mathcal{X}_\nu  \, : \, \|(u_3 + \alpha_1)^{-1}\|_\nu,\, \|(u_3 + \alpha_2)^{-1}\|_\nu,\, \|\partial_t u\|_{\mathcal{U}_\nu} < \infty \right\}.
\end{equation}

Our strategy is to show that a fixed-point operator (yet to be constructed) is a contraction around a finite Fourier--Chebyshev series that approximates the family of periodic orbits.
This approximation lives in a finite-dimensional subspace of $\mathcal{X}_\nu$.
To obtain this space, we introduce the \emph{truncation operators} $\Pi_K, \Pi_{N,K} : C^\infty(\mathbb{R}/2\pi\mathbb{Z} \times [-1, 1], \mathbb{C}) \to C^\infty(\mathbb{R}/2\pi\mathbb{Z} \times [-1, 1], \mathbb{C})$ by
\begin{equation} \label{index-K}
(\Pi_K u)_k
\bydef
\begin{cases}
u_k, & |k| \le K,  \\
0, & |k| > K,
\end{cases}, \qquad
(\Pi_{N,K} u)_k
\bydef
\begin{cases}
(\hat{\Pi}_N u_k)_n, & |k| \le K,  \\
0, & |k| > K,
\end{cases}
\end{equation}
with $\hat{\Pi}_N : C^\infty ([-1, 1], \mathbb{C}) \to C^\infty ([-1, 1], \mathbb{C})$ given by
\begin{equation} \label{index-N}
(\hat{\Pi}_N u)_n
\bydef
\begin{cases}
u_n, & n \le N,  \\
0, & n > N.
\end{cases}
\end{equation}
Moreover, we consider the \emph{tail operator} 
\begin{equation}
\Pi_{>K} \bydef I - \Pi_K.
\end{equation}
Note that both truncation operators $\Pi_K$ and $\Pi_{N, K}$ naturally extend to $\chi = (\tau, \zeta_1, \zeta_2, u) \in \mathcal{X}_\nu$ by acting component-wise
\begin{equation}
\Pi_K \chi = (\tau, \zeta_1, \zeta_2, \Pi_K u), \qquad
\Pi_{N, K} \chi = (\hat{\Pi}_N \tau, \hat{\Pi}_N \zeta_1, \hat{\Pi}_N \zeta_2, \Pi_{N, K} u).
\end{equation}

Fix $N, K \in \mathbb{N}$.
{
We consider
\begin{equation}
\partial_t \Gamma \in \Pi_{N, K} \mathcal{U}_\nu, \qquad
\ell(\psi) = [\Pi_{N, K} \psi](0, \,\cdot), \quad \text{for all }\psi \in \mathcal{W}_\nu.
\end{equation}
With this choice of $\partial_t \Gamma$ and $\ell$, the map $\rho$ given in\eqref{eq:rho} only depends on a finite number of Fourier--Chebyshev coefficients.
}

\subsection{Newton--Kantorovich type theorem}
\label{sec:contraction}

The next theorem provides sufficient conditions for $T$ to be a contraction which, by construction of $A$, is enough to obtain a zero of $F$.
{ The proof is standard and can be found, for instance, in~\cite{Bre25}.}

{
\begin{theorem}\label{thm:nk}
Let $R > 0$, $\bar{\chi} \in \Pi_{N,K} \mathcal{X}_\nu$, and $A : \mathcal{X}_\nu \to \mathcal{X}_\nu$ be an injective linear map.
Suppose there exist positive constants $Y, Z_1, Z_2 = Z_2(R)$ satisfying
\begin{subequations}\label{eq:bounds}
\begin{align}
\|AF(\bar{\chi}) \|_{\mathcal{X}_\nu} &\le Y, \\
\|ADF(\bar{\chi}) - I \|_{\mathscr{B}(\mathcal{X}_\nu, \mathcal{X}_\nu)} &\le Z_1, \\
\sup_{\chi \in B_R(\bar{\chi})}\|AD^2F(\chi) \|_{\mathscr{B}(\mathcal{X}_\nu, \mathscr{B}(\mathcal{X}_\nu, \mathcal{X}_\nu))} &\le Z_2,
\end{align}
\end{subequations}
where $B_R(\bar{\chi})$ is the closed ball in $\mathcal{X}_\nu$ centered at $\bar{\chi}$ with radius $R$, and $\mathscr{B}(\mathcal{X}, \mathcal{Y})$ denotes the set of bounded linear operators from $\mathcal{X}$ to $\mathcal{Y}$.

If
\begin{subequations}\label{eq:contraction_cond}
\begin{align}
2YZ_2 &\le (1 - Z_1)^2, \\
Z_1 &< 1,
\end{align}
\end{subequations}
then, for any $r > 0$ such that
\begin{equation}
\frac{1-Z_1 - \sqrt{(1-Z_1)^2 - 2YZ_2}}{Z_2} \le r < \min\left(\frac{1-Z_1}{Z_2}, R\right),
\end{equation}
there exists a unique zero $\tilde{\chi} \in B_r(\bar{\chi})$ of $F$.
\end{theorem}
}

Before verifying the contraction criteria given in the above theorem, we must first have an approximation of the branch of periodic orbits in $\Pi_{N, K} \mathcal{X}_\nu$, denoted by $\bar{\chi} = (\bar{\tau}, \bar{\zeta}_1, \bar{\zeta}_2, \bar{u})$.
This can be achieved by interpolating each Fourier mode of the approximate zeros $\bar{\chi}_0, \dots, \bar{\chi}_N$ of $F|_{\eta_\ell}$, where $F|_{\eta_\ell}$ is the mapping defined in~\eqref{eq:zero_finding_problem} and evaluated at $\eta = \eta_\ell$.
A practical way is to consider a finite truncation of $F|_{\eta_\ell}$ and use Newton's method to produce an accurate $\bar{\chi}_\ell$ at each of the $N + 1$ Chebyshev nodes
\begin{equation}
\eta_\ell \bydef -\cos\left(\frac{\ell\pi}{N}\right), \qquad \ell = 0, \dots, N.
\end{equation}
Then, the corresponding Fourier coefficients can be recovered by applying the inverse discrete Fourier transform (recall identity~\eqref{eq:cos_cheb}).
In practice, this step can be carried out efficiently by means of fast Fourier transform algorithms.

Moreover, we need to construct the operator $A \in \mathscr{B}(\mathcal{X}_\nu, \mathcal{X}_\nu)$.
We set
\begin{equation}
A \bydef A_\textnormal{finite} \Pi_K + A_\textnormal{tail} \Pi_{>K},
\end{equation}
with $A_\textnormal{tail} : \Pi_{>K} \mathcal{X}_\nu \to \Pi_{>K} \mathcal{X}_\nu$ defined, for all $\chi = (\tau, \zeta_1, \zeta_2, u) \in \mathcal{X}_\nu$, $u = (u_1, u_2, u_3)$, by
\begin{equation}\label{eq:tail}
A_\textnormal{tail} \chi \bydef (0, 0, 0, u'), \qquad
(u_j')_k \bydef
\begin{cases}
0, & |k| \le K, \\
\displaystyle \frac{1}{ik} (u_j)_k, & |k| > K,
\end{cases} \quad j = 1, 2, 3.
\end{equation}
On the other hand, $A_\textnormal{finite} \in \mathscr{B}(\Pi_K \mathcal{X}_\nu, \Pi_K \mathcal{X}_\nu)$ is an approximation of the inverse $(\Pi_K DF(\bar{\chi}) \Pi_K)^{-1}$ constructed by interpolating each entries of the $N+1$ matrices $(\Pi_{0,K} DF|_{\eta_\ell}(\bar{\chi}) \Pi_{0,K})^{-1}$ with a Chebyshev polynomial.

In the next sections, we derive practical estimates to verify the assumptions of Theorem~\ref{thm:nk}.
Specifically, we obtain explicit bounds for each norm in \eqref{eq:bounds}, expressed in terms of the numerical approximation $\bar{\chi}$ and so that only a finite, potentially substantial, set of computer calculations is needed.

\subsubsection{Estimate for $\|A F (\bar{\chi})\|_{\mathcal{X}_\nu}$}
\label{sec:Ybound}

To control $\|A F (\bar{\chi})\|_{\mathcal{X}_\nu}$, we begin by defining a finite-dimensional approximation $W_0 \in \Pi_{2N, K} \mathcal{X}_\nu$ of $F (\bar{\chi})$.
Let
\begin{equation}
W_0 \bydef
\begin{pmatrix}
\rho(\bar{u}) \\
\partial_t \bar{u} - \omega_0
\end{pmatrix},
\end{equation}
with $\omega_0 \in \Pi_{N, K} \mathcal{U}_\nu$ being an approximation of $\bar{\tau} f(\bar{\zeta}_1, \bar{\zeta}_2, \bar{u})$ found numerically.
The triangle inequality yields
\begin{equation}
\|A F (\bar{\chi})\|_{\mathcal{X}_\nu}
\le \|A W_0\|_{\mathcal{X}_\nu} +  \|A\|_{\mathscr{B}(\mathcal{X}_\nu, \mathcal{X}_\nu)} \|F (\bar{\chi}) - W_0\|_{\mathcal{X}_\nu},
\end{equation}
where
\begin{equation}\label{eq:W0}
\|F (\bar{\chi}) - W_0\|_{\mathcal{X}_\nu}
= \|\bar{\tau} f(\bar{\zeta}_1, \bar{\zeta}_2, \bar{u}) - \omega_0\|_{\mathcal{U}_\nu},
\end{equation}
and
\begin{equation}
\|A\|_{\mathscr{B}(\mathcal{X}_\nu, \mathcal{X}_\nu)} = \max\left( \|A_\textnormal{finite}\|_{\mathscr{B}(\mathcal{X}_\nu, \mathcal{X}_\nu)}, \frac{1}{K + 1}\right).
\end{equation}

While the products of functions are given by the discrete convolution formulas \eqref{eq:conv1} and \eqref{eq:conv2}, to bound $\|\bar{\tau} f(\bar{\zeta}_1, \bar{\zeta}_2, \bar{u}) - \omega_0\|_{\mathcal{U}_\nu}$ appearing in \eqref{eq:W0} it remains to handle the division of functions.

\begin{lemma}
Let $\bar{\phi}, \bar{\phi}_\textnormal{inv} \in \Pi_{N, K} \mathcal{W}_\nu$.
If $\|1 - \bar{\phi} \bar{\phi}_\textnormal{inv}\|_\nu < 1$, then
\begin{equation}
\| \bar{\phi}_\textnormal{inv} - \bar{\phi}^{-1} \|_\nu \le \frac{\|\bar{\phi}_\textnormal{inv} (1 - \bar{\phi} \bar{\phi}_\textnormal{inv})\|_\nu}{1 - \|1 - \bar{\phi} \bar{\phi}_\textnormal{inv}\|_\nu}.
\end{equation}
\end{lemma}

\begin{proof}
The proof is an application of Neumann series.
First, $\|1 - \bar{\phi} \bar{\phi}_\textnormal{inv}\|_\nu < 1$ implies that $\bar{\phi}$ and $\bar{\phi}_\textnormal{inv}$ are invertible. Then
\begin{align*}
\| \bar{\phi}_\textnormal{inv} - \bar{\phi}^{-1} \|_\nu
&= \| \bar{\phi}_\textnormal{inv} (1 - (\bar{\phi} \bar{\phi}_\textnormal{inv})^{-1} ) \|_\nu \\
&= \| \bar{\phi}_\textnormal{inv} (1 - (1 -( 1 - \bar{\phi} \bar{\phi}_\textnormal{inv}))^{-1} ) \|_\nu \\
&= \Big\| \bar{\phi}_\textnormal{inv} \left(1 - \sum_{n \ge 0} (1 - \bar{\phi} \bar{\phi}_\textnormal{inv})^n \right) \Big\|_\nu \\
&= \Big\| \bar{\phi}_\textnormal{inv} (1 - \bar{\phi} \bar{\phi}_\textnormal{inv}) \sum_{n \ge 0} (1 - \bar{\phi} \bar{\phi}_\textnormal{inv})^n \Big\|_\nu \\
&\le \| \bar{\phi}_\textnormal{inv} (1 - \bar{\phi} \bar{\phi}_\textnormal{inv}) \|_\nu \sum_{n \ge 0} \| 1 - \bar{\phi} \bar{\phi}_\textnormal{inv}\|_\nu^n \\
&= \frac{\| \bar{\phi}_\textnormal{inv} (1 - \bar{\phi} \bar{\phi}_\textnormal{inv}) \|_\nu}{1 - \| 1 - \bar{\phi} \bar{\phi}_\textnormal{inv}\|_\nu}.
\end{align*}
\end{proof}

\subsubsection{Estimate for $\|A DF (\bar{\chi}) - I\|_{\mathscr{B}(\mathcal{X}_\nu, \mathcal{X}_\nu)}$}
\label{sec:Z1bound}

For any element $\phi \in \mathcal{W}_\nu$, we define the associated multiplication operator $\mathcal{M}_\phi : \mathcal{W}_\nu \to \mathcal{W}_\nu$ by
\begin{equation}
[\mathcal{M}_\phi \psi](t, \eta) \bydef \phi (t, \eta) \psi (t, \eta), \qquad \text{for all } \psi \in \mathcal{W}_\nu.
\end{equation}
Since $\mathcal{P}_\nu$ can naturally be viewed as a subspace of $\mathcal{W}_\nu$ -- as constant functions are trivially periodic --, there is also a multiplication operator $\mathcal{M}_\phi : \mathcal{P}_\nu \to \mathcal{W}_\nu$ given by
\begin{equation}
[\mathcal{M}_\phi \psi](t, \eta) \bydef \phi (t, \eta) \psi (\eta), \qquad \text{for all } \psi \in \mathcal{P}_\nu.
\end{equation}
For simplicity, we use the same symbol $\mathcal{M}_\phi$ for both operators.
This slight abuse of notation should not cause confusion since the operators are essentially of the same nature.

We first consider a finite-rank operator $W_1$ approximating $DF(\bar{x})$.
Let
\begin{equation}
W_1 \bydef
\begin{pmatrix}
0 & 0 & 0 & \rho \\
-\omega_2 & -\omega_3 & -\omega_4 & \partial_t - \omega_1
\end{pmatrix},
\end{equation}
where $\omega_1 \in \mathscr{B}(\mathcal{U}_\nu, \mathcal{U}_\nu)$ and $\omega_2, \omega_3, \omega_4 \in \mathscr{B}(\mathcal{P}_\nu, \mathcal{U}_\nu)$ are such that
\begin{alignat}{2}
\omega_1 \phi &\bydef
\sum_{j = 1}^3 
\begin{pmatrix}
\omega_1^{(1,j)} \phi_j \\
\omega_1^{(2,j)} \phi_j \\
\omega_1^{(3,j)} \phi_j
\end{pmatrix}, \qquad &&\text{for all } \phi = (\phi_1, \phi_2, \phi_3) \in \mathcal{U}_\nu, \\
\omega_j \psi &\bydef
\begin{pmatrix}
\omega_j^{(1)} \psi \\
\omega_j^{(2)} \psi \\
\omega_j^{(3)} \psi
\end{pmatrix}, \qquad &&\text{for all } \psi \in \mathcal{P}_\nu, \quad j = 2, 3, 4,
\end{alignat}
with $\omega_1^{(i,j)}, \omega_2^{(i)}, \omega_3^{(i)}, \omega_4^{(i)} \in \Pi_{N, K} \mathcal{W}_\nu$ for $i, j = 1, 2, 3$.
In other words, $\omega_1$ can be visualized as the 3-by-3 matrix
\begin{equation*}
\omega_1
=
\begin{pmatrix}
\mathcal{M}_{\omega_1^{(1,1)}} & \mathcal{M}_{\omega_1^{(1,2)}} & \mathcal{M}_{\omega_1^{(1,3)}} \\
\mathcal{M}_{\omega_1^{(2,1)}} & \mathcal{M}_{\omega_1^{(2,2)}} & \mathcal{M}_{\omega_1^{(2,3)}} \\
\mathcal{M}_{\omega_1^{(3,1)}} & \mathcal{M}_{\omega_1^{(3,2)}} & \mathcal{M}_{\omega_1^{(3,3)}}
\end{pmatrix},
\end{equation*}
and $\omega_j$, for $j = 2, 3, 4$, as the 3-component vector
\begin{equation*}
\omega_j
=
\begin{pmatrix}
\mathcal{M}_{\omega_j^{(1)}} \\
\mathcal{M}_{\omega_j^{(2)}} \\
\mathcal{M}_{\omega_j^{(3)}}
\end{pmatrix}.
\end{equation*}
The objects $\omega_1$, $\omega_2$, $\omega_3$, and $\omega_4$ are numerical approximations for $\bar{\tau} \partial_u f (\bar{\zeta}_1, \bar{\zeta}_2, \bar{u})$, $f(\bar{\zeta}_1, \bar{\zeta}_2, \bar{u})$, $\bar{\tau} \partial_{\zeta_1} f(\bar{\zeta}_1, \bar{\zeta}_2, \bar{u})$, and $\bar{\tau} \partial_{\zeta_2} f(\bar{\zeta}_1, \bar{\zeta}_2, \bar{u})$, respectively.

The triangle inequality yields
\begin{equation}
\|A DF(\bar{\chi}) - I\|_{\mathscr{B}(\mathcal{X}_\nu, \mathcal{X}_\nu)}
\le \|A W_1 - I\|_{\mathscr{B}(\mathcal{X}_\nu, \mathcal{X}_\nu)} + \| A \|_{\mathscr{B}(\mathcal{X}_\nu, \mathcal{X}_\nu)} \| DF(\bar{\chi}) - W_1 \|_{\mathscr{B}(\mathcal{X}_\nu, \mathcal{X}_\nu)}.
\end{equation}
It is straightforward to check that
\begin{align}
\begin{split}
\| DF(\bar{\chi}) - W_1 \|_{\mathscr{B}(\mathcal{X}_\nu, \mathcal{X}_\nu)}
\le \max\Bigg(
&\sum_{j=1}^3 \|f_j(\bar{\zeta}_1, \bar{\zeta}_2, \bar{u}) - \omega_2^{(j)}\|_\nu,
\sum_{j=1}^3 \|\bar{\tau} \partial_{\zeta_1} f_j(\bar{\zeta}_1, \bar{\zeta}_2, \bar{u}) - \omega_3^{(j)}\|_\nu, \\
&\sum_{j=1}^3 \|\bar{\tau} \partial_{\zeta_2} f_j(\bar{\zeta}_1, \bar{\zeta}_2, \bar{u}) - \omega_4^{(j)}\|_\nu,
\max_{1\le j\le 3} \sum_{i=1}^3 \|\bar{\tau} \partial_{u_j} f_i(\bar{\zeta}_1, \bar{\zeta}_2, \bar{u}) - \omega_1^{(i,j)}\|_\nu
\Bigg),
\end{split}
\end{align}
where we used the fact that, for all $\phi \in \mathcal{W}_\nu$, we have $\|\mathcal{M}_\phi\|_{\mathscr{B}(\mathcal{P}_\nu, \mathcal{W}_\nu)} = \|\mathcal{M}_\phi\|_{\mathscr{B}(\mathcal{W}_\nu, \mathcal{W}_\nu)} = \|\phi\|_\nu$.

To control $\|A W_1 - I\|_{\mathscr{B}(\mathcal{X}_\nu, \mathcal{X}_\nu)}$, we note that $W_1$ is a block-wise banded operator with bandwidth $N K$, so it is practical to consider
\begin{equation}
\|A W_1 - I\|_{\mathscr{B}(\mathcal{X}_\nu, \mathcal{X}_\nu)}
= \max \Big(\|A W_1 \Pi_{2K} - \Pi_{2K}\|_{\mathscr{B}(\mathcal{X}_\nu, \mathcal{X}_\nu)}, \|A W_1 \Pi_{>2K} - \Pi_{>2K}\|_{\mathscr{B}(\mathcal{X}_\nu, \mathcal{X}_\nu)} \Big).
\end{equation}
Indeed, from the banded structure we have $W_1 \Pi_{2K} = \Pi_{3K} W_1 \Pi_K$, and it follows that
\begin{align}
\|A W_1 \Pi_{2K} - \Pi_{2K}\|_{\mathscr{B}(\mathcal{X}_\nu, \mathcal{X}_\nu)}
&= \|\Pi_{3K} A \Pi_{3K} W_1 \Pi_{2K} - \Pi_{2K}\|_{\mathscr{B}(\mathcal{X}_\nu, \mathcal{X}_\nu)} \nonumber\\
&= \|\Pi_{2N, 3K} A \Pi_{N, 3K} W_1 \Pi_{0, 2K} - \Pi_{0, 2K}\|_{\mathscr{B}(\mathcal{X}_\nu, \mathcal{X}_\nu)},
\end{align}
where, to obtain the second equality, we used the fact that $\Pi_{3K} A \Pi_{3K} W_1 \Pi_{2K} - \Pi_{2K}$ acts as a multiplication operator with respect to $\eta$.
On the other hand, using once more the banded structure of $W_1$, we have $\Pi_K W_1 \Pi_{>2K} = 0$, so that
\begin{align}
\|A W_1 \Pi_{>2K} - \Pi_{>2K}\|_{\mathscr{B}(\mathcal{X}_\nu, \mathcal{X}_\nu)}
&= \|A_\textnormal{tail} \Pi_{>K} W_1 \Pi_{>2K} - \Pi_{>2K}\|_{\mathscr{B}(\mathcal{X}_\nu, \mathcal{X}_\nu)} \nonumber\\
&\le \|A_\textnormal{tail}\|_{\mathscr{B}(\mathcal{X}_\nu, \mathcal{X}_\nu)} \|\Pi_{>K} \omega_1 \Pi_{>2K}\|_{\mathscr{B}(\mathcal{U}_\nu, \mathcal{U}_\nu)} \nonumber\\
&\le \frac{1}{K+1} \max_{1 \le j \le 3} \sum_{i = 1}^3 \|\omega_1^{(i,j)}\|_\nu.
\end{align}
%

\subsubsection{Estimate for $\sup_{\chi \in B_R(\bar{\chi})} \|A D^2 F(\chi)\|_{\mathscr{B}(\mathcal{X}_\nu, \mathscr{B}(\mathcal{X}_\nu, \mathcal{X}_\nu))}$}
\label{sec:Z2bound}

Since
\begin{equation*}
\sup_{\chi \in B_R(\bar{\chi})} \|A D^2 F(\chi)\|_{\mathscr{B}(\mathcal{X}_\nu, \mathscr{B}(\mathcal{X}_\nu, \mathcal{X}_\nu))}
\le \| A \|_{\mathscr{B}(\mathcal{X}_\nu, \mathcal{X}_\nu)} \sup_{\chi \in B_R(\bar{\chi})} \|D^2 F(\chi)\|_{\mathscr{B}(\mathcal{X}_\nu, \mathscr{B}(\mathcal{X}_\nu, \mathcal{X}_\nu))},
\end{equation*}
it remains to bound $\|D^2 F(\chi)\|_{\mathscr{B}(\mathcal{X}_\nu, \mathscr{B}(\mathcal{X}_\nu, \mathcal{X}_\nu))}$ for all $\chi \in B_R(\bar{\chi})$.
Introducing the notation $\nabla = (\partial_\tau, \partial_{\zeta_1}, \partial_{\zeta_2}, \partial_{u_1}, \partial_{u_2}, \partial_{u_3})$, and to simplify the notation $f_j = f_j(\zeta_1, \zeta_2, u)$, we have
\begin{align*}
\begin{split}
\nabla (\tau f_1) &=
\Bigg(
f_1, 0, 0, \tau \delta_1 \frac{u_3 - \lambda_1}{u_3 + \alpha_1}, 0, \tau \delta_1 \frac{\lambda_1 + \alpha_1}{(u_3 + \alpha_1)^2} u_1
\Bigg), \\
\nabla (\tau f_2) &=
\Bigg(
f_2, 0, 0, 0, \tau \delta_2 \frac{u_3 - \lambda_2}{u_3 + \alpha_2}, \tau \delta_2 \frac{\lambda_2 + \alpha_2}{(u_3 + \alpha_2)^2} u_2
\Bigg), \\
\nabla (\tau f_3) &=
\Bigg(
f_3, -\tau \frac{u_1}{u_3 + \alpha_1} u_3, -\tau \frac{u_2}{u_3 + \alpha_2} u_3, -\tau \zeta_1 \frac{u_3}{u_3 + \alpha_1}, -\tau \zeta_2 \frac{u_3}{u_3 + \alpha_2}, \tau \Big[ 1 - 2u_3 - \sum_{j=1}^2 \zeta_j u_j \frac{\alpha_j}{(u_3 + \alpha_j)^2} \Big]
\Bigg).
\end{split}
\end{align*}
Hence, writing $F = (F_1, F_2, F_3, F_4, F_5, F_6)$ and $\nabla = (\nabla_1, \nabla_2, \nabla_3, \nabla_4, \nabla_5, \nabla_6)$, we obtain
\begin{alignat}{2}\label{eq:Z2}
\|D^2 F(\chi)\|_{\mathscr{B}(\mathcal{X}_\nu, \mathscr{B}(\mathcal{X}_\nu, \mathcal{X}_\nu))}
&\le \max_{1 \le j \le 6} &&\sum_{i = 3}^6 \| \nabla_j \nabla F_i (\chi) \| \nonumber\\
&= \max \Bigg(
&&\sum_{j=1}^3 \| \partial_\tau \nabla (\tau f_j) \|,
\| \partial_{\zeta_1} \nabla (\tau f_3) \|,
\| \partial_{\zeta_2} \nabla (\tau f_3) \|, \nonumber\\
& &&\| \partial_{u_1} \nabla (\tau f_1) \| + \| \partial_{u_1} \nabla (\tau f_3) \|,
\| \partial_{u_2} \nabla (\tau f_2) \| + \| \partial_{u_2} \nabla (\tau f_3) \|,
\sum_{j = 1}^3 \| \partial_{u_3} \nabla (\tau f_j) \|
\Bigg),
\end{alignat}
where, for $j = 1, 2$,
\begin{align*}
\| \partial_\tau \nabla (\tau f_j) \| &= \max\left(
\| \delta_j \frac{u_3 - \lambda_j}{u_3 + \alpha_j} \|_\nu, \| \delta_j \frac{\lambda_j + \alpha_j}{(u_3 + \alpha_j)^2} u_j \|_\nu
\right), \\
\| \partial_{u_j} \nabla (\tau f_j) \| &= \max\left(
\| \delta_j \frac{u_3 - \lambda_j}{u_3 + \alpha_j} \|_\nu, \| \tau \delta_j \frac{\lambda_j + \alpha_j}{(u_3 + \alpha_j)^2} \|_\nu
\right), \\
\| \partial_{u_3} \nabla (\tau f_j) \| &= \max\left(
\| \delta_j \frac{\lambda_j + \alpha_j}{(u_3 + \alpha_j)^2} u_j \|_\nu, \| \tau \delta_j \frac{\lambda_j + \alpha_j}{(u_3 + \alpha_j)^2} \|_\nu, 2 \| \tau \delta_j \frac{\lambda_j + \alpha_j}{(u_3 + \alpha_j)^3} u_j \|_\nu
\right), \\
\| \partial_{\zeta_j} \nabla (\tau f_3) \| &= \max\left(
\| \frac{u_j u_3}{u_3 + \alpha_j} \|_\nu, \| \tau \frac{u_3}{u_3 + \alpha_j} \|_\nu, \| \tau u_j \frac{\alpha_j}{(u_3 + \alpha_j)^2} \|_\nu
\right), \\
\| \partial_{u_j} \nabla (\tau f_3) \| &= \max\left(
\| \frac{\zeta_j u_3}{u_3 + \alpha_j} \|_\nu, \| \tau \frac{u_3}{u_3 + \alpha_j} \|_\nu, \| \tau \zeta_j \frac{\alpha_j}{(u_3 + \alpha_j)^2} \|_\nu
\right),
\end{align*}
and
\begin{alignat*}{2}
\| \partial_\tau \nabla (\tau f_3) \| &= \max\Bigg(
&&\| \frac{u_1 u_3}{u_3 + \alpha_1} \|_\nu, \| \frac{u_2 u_3}{u_3 + \alpha_2} \|_\nu, \| \zeta_1 \frac{u_3}{u_3 + \alpha_1} \|_\nu, \| \zeta_2 \frac{u_3}{u_3 + \alpha_2} \|_\nu, \| 1 -2 u_3- \sum_{j=1}^2 \zeta_j u_j \frac{\alpha_j}{(u_3 + \alpha_j)^2} \|_\nu
\Bigg), \\
\| \partial_{u_3} \nabla (\tau f_3) \| &= \max\Bigg(
&&\| 1 -2u_3 - \sum_{j=1}^2 \zeta_j u_j \frac{\alpha_j}{(u_3 + \alpha_j)^2} \|_\nu,
\| \tau \frac{u_1 \alpha_1}{(u_3 + \alpha_1)^2} \|_\nu, \| \tau \frac{u_2 \alpha_2}{(u_3 + \alpha_2)^2} \|_\nu, \\
& &&\| \tau \frac{\zeta_1 \alpha_1}{(u_3 + \alpha_1)^2} \|_\nu, \| \tau \frac{\zeta_2 \alpha_2}{(u_3 + \alpha_2)^2} \|_\nu,
2\| \tau \Big[-1 + \sum_{j=1}^2 \zeta_j u_j \frac{\alpha_j}{(u_3 + \alpha_j)^3} \Big] \|_\nu
\Bigg).
\end{alignat*}

To compute the above norms, we need to bound $(u_3 + \alpha_j)^{-p}$ for $p = 1, 2, 3$ and $u_3 \in B_R(\bar{u}_3)$.
The following lemma provides a means for $p = 1$, while the cases $p = 2, 3$ are handled by using the Banach algebra properties given in \eqref{eq:banach_algebra1} and \eqref{eq:banach_algebra2}.

\begin{lemma}\label{lem:Z2_inv}
Let $R > 0$ and $\bar{\phi}, \bar{\phi}_\textnormal{inv} \in \Pi_{N, K} \mathcal{W}_\nu$.
If $\|1 - \bar{\phi} \bar{\phi}_\textnormal{inv}\|_\nu + R \|\bar{\phi}_\textnormal{inv}\|_\nu < 1$, then
\begin{equation}
\sup_{\phi \in (B_R (\bar{\phi})} \| \phi^{-1} \|_\nu
\le \frac{\|\bar{\phi}_\textnormal{inv}\|_\nu}{1 - \|1 - \bar{\phi} \bar{\phi}_\textnormal{inv}\|_\nu - R \|\bar{\phi}_\textnormal{inv}\|_\nu}.
\end{equation}
\end{lemma}

\begin{proof}
Using Neumann series, we have
\begin{align*}
\sup_{\phi \in B_R (\bar{\phi})} \|\phi^{-1}\|_\nu
&= \sup_{\phi \in B_R (\bar{\phi})} \|\bar{\phi}_\textnormal{inv} (\phi \bar{\phi}_\textnormal{inv})^{-1}\|_\nu \\
&= \sup_{\phi \in B_R (\bar{\phi})} \|\bar{\phi}_\textnormal{inv} (1 - (1 -\phi \bar{\phi}_\textnormal{inv}))^{-1}\|_\nu \\
&= \sup_{\phi \in B_R (\bar{\phi})} \|\bar{\phi}_\textnormal{inv} \sum_{n \ge 0} (1 -\phi \bar{\phi}_\textnormal{inv})^n\|_\nu \\
&\le \|\bar{\phi}_\textnormal{inv} \|_\nu \sup_{\phi \in B_R (\bar{\phi})} \sum_{n \ge 0} \|1 -\phi \bar{\phi}_\textnormal{inv}\|_\nu^n \\
&\le \|\bar{\phi}_\textnormal{inv} \|_\nu \sum_{n \ge 0} \big(\|1 - \bar{\phi} \bar{\phi}_\textnormal{inv}\|_\nu + R \|\bar{\phi}_\textnormal{inv}\|_\nu\big)^n \\
&= \frac{\|\bar{\phi}_\textnormal{inv} \|_\nu}{1 - \| 1 - \bar{\phi} \bar{\phi}_\textnormal{inv}\|_\nu - R \|\bar{\phi}_\textnormal{inv}\|_\nu}.
\end{align*}
\end{proof}

Now, the estimate \eqref{eq:Z2} is obtained by using the triangle inequality and Lemma \ref{lem:Z2_inv} repeatedly.
To illustrate, for $j=1, 2$, we get
\begin{align*}
\sup_{\chi = (\tau, \zeta_1, \zeta_2, u) \in B_R(\bar{\chi})} \| \partial_\tau \nabla (\tau f_j) \|
&= \max\left(
\| \delta_j \frac{u_3 - \lambda_j}{u_3 + \alpha_j} \|_\nu, \| \delta_j \frac{\lambda_j + \alpha_j}{(u_3 + \alpha_j)^2} u_j \|_\nu
\right) \\
&\le \delta_j \max\left(
(\| \bar{u}_3 - \lambda_j \|_\nu  + R) \mathfrak{C}_j, (\lambda_j + \alpha_j) (\|\bar{u}_j\|_\nu + R) \mathfrak{C}_j^2
\right),
\end{align*}
where $\mathfrak{C}_j$ is obtained using Lemma \ref{lem:Z2_inv}, namely, given a finite approximation, denoted $\bar{\phi}_{\textnormal{inv},j}$, of the inverse of $\bar{u}_3 + \alpha_j$, we obtain
\begin{equation*}
\mathfrak{C}_j \ge \frac{\|\bar{\phi}_{\textnormal{inv},j}\|_\nu}{1 - \|1 - (\bar{u}_3 + \alpha_j)\bar{\phi}_{\textnormal{inv},j}\|_\nu - \|\bar{\phi}_{\textnormal{inv},j}\|_\nu R}.
\end{equation*}
%

\subsection{Real-valued solutions and boundary crossing}
\label{sec:aposteriori}

Given $\bar{\chi} \in \Pi_{N, K} \mathcal{X}_\nu$, if Theorem \ref{thm:nk} holds, then there exists a unique $\tilde{\chi} = (\tilde{\tau}, \tilde{\zeta}_1, \tilde{\zeta}_2, \tilde{u}) \in B_r(\bar{\chi})$, for some $r > 0$, satisfying $F(\tilde{\chi}) = 0$.
Whence, from \eqref{eq:kappa}, we have
\begin{equation}
\eta (\kappa) = \frac{2 \kappa_1 \kappa_2 - \kappa(\kappa_1 + \kappa_2)}{\kappa(\kappa_1 - \kappa_2)},
\end{equation}
such that, for any given $\kappa \in [\kappa_1, \kappa_2]$,
\begin{equation}
\begin{cases}
X_j (t, \kappa) =
\displaystyle \kappa \tilde{\zeta}_j (\eta(\kappa)) \frac{\gamma y_j}{m_j} \tilde{u}_j \left(\frac{\gamma}{\tilde{\tau}(\eta(\kappa))} t, \eta(\kappa)\right), & j =1, 2, \\
S (t, \kappa) \hspace{0.2cm}=
\displaystyle \kappa \tilde{u}_3 \left(\frac{\gamma}{\tilde{\tau}(\eta(\kappa))} t, \eta(\kappa)\right),
\end{cases}
\end{equation}
is a $2\pi\tau(\eta(\kappa))$-periodic solution to the original system \eqref{eq:original_model} as desired.

However, functions in $\mathcal{X}_\nu$ are in principle complex-valued.
So we still need to verify that we can obtain a zero $\tilde{\chi}$ of $F$ which is real-valued.
Such a property follows from a complex conjugacy symmetry.
Consider the linear operator $\Sigma : \mathcal{X}_\nu \to \mathcal{X}_\nu$ defined by, for all $\chi = (\tau, \zeta_1, \zeta_2, u) \in \mathcal{X}_\nu$, $u = (u_1, u_2, u_3)$,
\begin{equation}
\Sigma (\chi) \bydef (\tau^*, \zeta_1^*, \zeta_2^*, \Sigma_0 (u_1), \Sigma_0 (u_2), \Sigma_0 (u_3)),
\end{equation}
where the superscript $^*$ denotes to the complex conjugate, i.e., for a function $\psi \in C^\infty ([-1, 1], \mathbb{C})$,
\begin{equation}
(\psi^*)_n = \psi_n^*, \qquad n \in \mathbb{N}_0,
\end{equation}
and where, for a function $\phi \in C^\infty (\mathbb{R}/2\pi\mathbb{Z} \times [-1, 1], \mathbb{C})$,
\begin{equation}
( \Sigma_0 (\phi) )_{n, k} \bydef \phi_{n, -k}^*, \qquad n  \in \mathbb{N}_0, \quad k \in \mathbb{Z}.
\end{equation}

Clearly, if $\Sigma (\chi) = \chi$ for $\chi = (\tau, \zeta_1, \zeta_2, u) \in \mathcal{X}_\nu$, then $\tau(\eta), \zeta_1 (\eta), \zeta_2 (\eta), u(t, \eta) \in \mathbb{R}$ for all $t \in \mathbb{R}, \eta \in [-1, 1]$.
It turns out that it is sufficient for the approximate zero $\bar{\chi}$ to satisfy $\Sigma (\bar{\chi}) = \bar{\chi}$, to ensure that the true solution $\tilde{\chi}$ also verifies it.

\begin{lemma}
Suppose that Theorem~\ref{thm:nk} holds, that is, there exist $\bar{\chi} \in \Pi_{N, K} \mathcal{X}_\nu$, and a unique $\tilde{\chi} \in B_r (\bar{\chi})$, for some $r > 0$, such that $F(\tilde{\chi}) = 0$.
If $\Sigma (\bar{\chi}) = \bar{\chi}$, then $\Sigma (\tilde{\chi}) = \tilde{\chi}$.
\end{lemma}

\begin{proof}
We have that $F(\Sigma (\tilde{\chi})) = \Sigma (F(\tilde{\chi})) = \Sigma (0) = 0$ and $\|\Sigma (\tilde{\chi}) - \bar{\chi}\|_{\mathcal{X}_\nu} = \|\Sigma (\tilde{\chi}) - \Sigma (\bar{\chi})\|_{\mathcal{X}_\nu} = \|\Sigma (\tilde{\chi} - \bar{\chi})\|_{\mathcal{X}_\nu} = \|\tilde{\chi} - \bar{\chi}\|_{\mathcal{X}_\nu} \le r$.
Hence, by the local uniqueness of $\tilde{\chi}$, it follows that $\Sigma (\tilde{\chi}) = \tilde{\chi}$ as desired.
\end{proof}

Furthermore, we must check that $\tilde{\tau} > 0$ and that we reach the boundary planes $Q_1$ and $Q_2$.
The following lemma shows how to rigorously evaluate $\tilde{\tau}$, $\tilde{\zeta}_1$, and $\tilde{\zeta}_2$.

\begin{lemma}\label{lem:enclosure-order0}
Let $\bar{\chi} \in \Pi_{N, K} \mathcal{X}_\nu$, $r > 0$ and $\tilde{\chi} = (\tilde{\tau}, \tilde{\zeta}_1, \tilde{\zeta}_2, \tilde{u})\in B_r (\bar{\chi})$.
For any $\eta \in [-1, 1]$, we have
\begin{equation}
|\tilde{\tau} (\eta) - \bar{\tau} (\eta)| \le r,\qquad
|\tilde{\zeta}_j (\eta) - \bar{\zeta}_j (\eta)| \le r, \quad j = 1, 2.
\end{equation}
\end{lemma}

\begin{proof}
For any function $\psi \in \mathcal{P}_\nu$, we have $\sup_{\eta \in [-1, 1]} |\psi(\eta)| \le \| \psi \|_\nu$.
\end{proof}

The previous lemma can only provide an enclosure of the exact value.
So the sign of $\tilde{\zeta}_j$ becomes undetermined as we approach $\hat{\eta}_j$, assuming there exists $\hat{\eta}_1, \hat{\eta}_2 \in (-1, 1)$ such that $\tilde{\zeta}_2 (\hat{\eta}_1) = 0$ and $\tilde{\zeta}_1 (\hat{\eta}_2) = 0$.
Let $h_1^-, h_1^+, h_2^-$, and $h_2^+$ satisfy $-1 < h_1^- < h_1^+ \le 0 \le h_2^- < h_2^+ < 1$.
Hence, to conclude that the family crosses the boundary plane $Q_1$ (resp., $Q_2$) only at some $\hat{\eta}_1 \in (h_1^-, h_1^+)$ (resp., $\hat{\eta}_2 \in (h_2^-, h_2^+)$), we need to show:
\begin{itemize}
\item Positive periodic solutions in the interior: $\tilde{\zeta}_1(\eta), \tilde{\zeta}_2(\eta) > 0$ for all $\eta \in [h_1^+, h_2^-]$;
\item Crossing of $Q_1$: $\tilde{\zeta}_2(h_1^-) < 0$ and $\frac{\mathrm{d}}{\mathrm{d}\eta} \tilde{\zeta}_2 (\eta) > 0$ for all $\eta \in [h_1^-, h_2^+]$;
\item Crossing of $Q_2$: $\tilde{\zeta}_1(h_2^+) < 0$ and $\frac{\mathrm{d}}{\mathrm{d}\eta} \tilde{\zeta}_1 (\eta) < 0$ for all $\eta \in [h_2^-, h_2^+]$.
\end{itemize}

The next lemma provides a way to retrieve the sign of the derivative for a function in $\mathcal{P}_\nu$, such as $\tilde{\zeta}_1$ and $\tilde{\zeta}_2$.

\begin{lemma}\label{lem:enclosure-order1}
Let $\bar{\psi} \in \Pi_N \mathcal{P}_\nu$, $\nu > 1$, $r > 0$ and $\tilde{\psi} \in B_r (\bar{\psi})$.
It follows that
\begin{equation}\label{eq:sign}
\textnormal{sign} \left(\frac{\mathrm{d}}{\mathrm{d} \eta} \tilde{\psi} (\eta)\right)
= \textnormal{sign} \left(\sum_{n\ge1} n \tilde{\psi}_n \sin (n \cos^{-1}(\eta))\right), \qquad \eta \in (-1, 1),
\end{equation}
where the infinite sum in the right-hand side satisfies
\begin{equation} \label{equation-69}
\Bigg|\sum_{n \ge 1} n \tilde{\psi}_n \sin (n \cos^{-1}(\eta))
- \sum_{n = 1}^N n \bar{\psi}_n \sin (n \cos^{-1}(\eta))\Bigg| \le r \Bigg(
\sum_{n = 1}^N n |\sin (n \cos^{-1}(\eta))| +
\frac{(\nu - 1)N + \nu}{(\nu - 1)^2 \nu^N}\Bigg).
\end{equation}
\end{lemma}

\begin{proof}
Using the identity \eqref{eq:cos_cheb} with $\theta \in [0, \pi]$, we obtain the relation
\begin{equation*}
\left(\frac{\mathrm{d}}{\mathrm{d} \eta}\Big|_{\eta = \cos(\theta)} \tilde{\psi} (\eta)\right) \sin(\theta) = 2 \sum_{n \ge 1} n \tilde{\psi}_n \sin (n \theta),
\end{equation*}
which proves \eqref{eq:sign}.
Moreover, since $\bar{\psi} \in \Pi_N \mathcal{P}_\nu$, we have for $n > N$ that
\begin{equation*}
|\tilde{\psi}_n| \nu^n
\le \sum_{|n| \le N} |\tilde{\psi}_{|n|} - \bar{\psi}_{|n|}| \nu^{|n|} + \sum_{|n| > N} |\tilde{\psi}_{|n|}| \nu^{|n|}
= \| \tilde{\psi} - \bar{\psi} \|_\nu
\le r.
\end{equation*}
Therefore, $|\tilde{\psi}_n| \le r \nu^{-n}$ for $n > N$, and, since $\nu > 1$, we have
\begin{equation*}
\sum_{n > N} \frac{n}{\nu^n} = \frac{(\nu - 1)N + \nu}{(\nu - 1)^2 \nu^N}.
\end{equation*}
\end{proof}



\section{Proof of stability}
\label{sec:stability}

We assume that Theorem \ref{thm:nk} holds such that there exists a locally unique $\tilde{\chi} = (\tilde{\tau}, \tilde{\zeta}_1, \tilde{\zeta}_2, \tilde{u}) \in \mathcal{X}_\nu$ satisfying $F (\tilde{\chi}) = 0$.
In this section, we show how to solve the eigenproblem for the whole branch of periodic orbits, and provide a sufficient criterion for such a family to be comprised of stable periodic orbits.
The underlying argument of the proof is similar to the one used for the proof of existence in Section \ref{sec:continuation}, in that it hinges on a local contraction argument.
However, the boundary crossing (where transcritical bifurcations occur) requires us to verify stability by different means when the branch is ``close to'' and ``far from'' the boundary; the adjectives ``close'' and ``far'' are to be quantified in this section.

From Floquet theory, the local stability of a periodic orbit to the auxiliary system \eqref{eq:model_blowup} is determined by studying the linearization of the vector field around it:
\begin{equation}\label{eq:linearization}
\partial_t v = \tilde{\tau} \partial_u f(\tilde{\zeta}_1, \tilde{\zeta}_2, \tilde{u}) v,
\end{equation}
where
\begin{equation}
\partial_u f(\tilde{\zeta}_1, \tilde{\zeta}_2, \tilde{u})
=
\begin{pmatrix}
\displaystyle \delta_1 \left(\frac{\tilde{u}_3 - \lambda_1}{\tilde{u}_3 + \alpha_1}\right) & 0 & \displaystyle \delta_1 \left(\frac{\alpha_1 + \lambda_1}{(\tilde{u}_3 + \alpha_1)^2}\right) \tilde{u}_1 \\
0 & \displaystyle \delta_2 \left(\frac{\tilde{u}_3 - \lambda_2}{\tilde{u}_3 + \alpha_2}\right) & \displaystyle \delta_2 \left(\frac{\alpha_2 + \lambda_2}{(\tilde{u}_3 + \alpha_2)^2}\right) \tilde{u}_2 \\
\displaystyle - \zeta_1 \frac{\tilde{u}_3}{\tilde{u}_3 + \alpha_1} & \displaystyle - \zeta_2 \frac{\tilde{u}_3}{\tilde{u}_3 + \alpha_2} & \displaystyle 1 - 2 \tilde{u}_3 - \alpha_1 \zeta_1 \frac{\tilde{u}_1}{(\tilde{u}_3 + \alpha_1)^2} - \alpha_2 \zeta_2 \frac{\tilde{u}_2}{(\tilde{u}_3 + \alpha_2)^2}
\end{pmatrix}.
\end{equation}

Let $\textnormal{Mat}_{m \times p} (\mathcal{R})$ be the set of $m$-by-$p$ matrices over the commutative ring $\mathcal{R}$.
The fundamental matrix solution to the linearized system \eqref{eq:linearization} can be expressed in the form $$\Phi(t) = V(t) e^{C t},$$ where $V : \mathbb{R} \to \textnormal{Mat}_{3\times3} (\mathbb{R})$ is $2\pi$-periodic, $C \in \textnormal{Mat}_{3\times3} (\mathbb{C})$, and $V(0) = I$.
The Floquet exponents $\mu_0, \mu_1, \mu_2$ are the eigenvalues of $C$.
Since $\partial_t u$ satisfies \eqref{eq:linearization}, corresponding to the trivial Floquet exponent $\mu_0 = 0$, it remains to find two more eigenvalues $\mu_1$ and $\mu_2$.
If their real part satisfy $\Re(\mu_1) < 0$ and $\Re(\mu_2) < 0$, then the periodic orbit is stable.
It is straightforward to show that the real part of the Floquet exponents associated with the original system \eqref{eq:original_model} have the same signs as the ones obtained by studying \eqref{eq:linearization}.
Indeed, the Floquet exponents of both systems are equal up to a scaling by $\tilde{\tau}/\gamma$; see the change of variables introduced in \eqref{eq:normalization}.

Now, along the branch of periodic orbits, the eigenvalues vary and changes in the algebraic and geometric multiplicity can occur.
This suggests that formulating a zero-finding problem to solve for the eigenpairs along the whole branch is not generally appropriate.
Instead, we will find $C$ and $V$ directly by solving the initial value problem
\begin{equation}\label{eq:floquet_ivp}
\begin{cases}
\partial_t V + V C = \tilde{\tau} \partial_u f(\tilde{\zeta}_1, \tilde{\zeta}_2, \tilde{u}) V, \\
V(0) = I_{3\times 3}.
\end{cases}
\end{equation}
To do so, consider the Banach space
\begin{equation}
\mathcal{Y}_\nu \bydef \textnormal{Mat}_{3 \times 3}(\mathcal{P}_\nu) \times \textnormal{Mat}_{3 \times 3}(\mathcal{U}_\nu),
\end{equation}
endowed with the norm given by
\begin{equation}
\| \upsilon \|_{\mathcal{Y}_\nu} \bydef \sum_{i,j=1}^3 \| C_{i,j} \|_\nu + \sum_{i=1}^3 \sum_{j=1}^3 \| V_{i,j} \|_{\mathcal{U}_\nu}, \qquad \text{for all } \upsilon = (C, V) \in \mathcal{Y}_\nu.
\end{equation}
The truncation operators $\Pi_K, \Pi_{N,K}$ naturally extend to $\mathcal{Y}_\nu$ by acting component-wise.
Then, the Floquet normal form corresponds to a zero of the mapping $G : \mathcal{D}(G) \subset \mathcal{Y}_\nu \to \mathcal{Y}_\nu$ given by
\begin{equation}\label{eq:G}
G (C, V)
\bydef
\begin{pmatrix}
\Pi_{N,K} V(0) - I_{3\times 3} \\
\partial_t V + V C - \tilde{\tau} \partial_u f(\tilde{\zeta}_1, \tilde{\zeta}_2, \tilde{u}) V
\end{pmatrix}.
\end{equation}

If $G (\tilde{C}, \tilde{V}) = 0$, then we can control the spectrum of $\tilde{C}$ by using the Gershgorin circle theorem.
However, as we cross the boundary planes $Q_1, Q_2$ at respectively $\hat{\eta}_1, \hat{\eta}_2$, one eigenvalue crosses the imaginary axis through the origin $0 \in \mathbb{C}$ due to the transcritical bifurcation.
Thus, in a neighbourhood of $\hat{\eta}_j$, the sign of the real part of this ``crossing eigenvalue'' is undetermined.
Recall that in Section \ref{sec:aposteriori} we have considered a slightly smaller interval $[h_1^-, h_2^+] \subset [-1, 1]$ which we partitioned into the three pieces:
\begin{enumerate}
\item[(i)] $[h_1^-, h_1^+]$ where the family crosses $Q_1$ at $\hat{\eta}_1 \in (h_1^-, h_1^+)$;
\item[(ii)] $[h_1^+, h_2^-]$ where the family is in the interior $\mathbb{R}^3_+$;
\item[(iii)] $[h_2^-, h_2^+]$ where the family crosses $Q_2$ at $\hat{\eta}_2 \in (h_2^-, h_2^+)$.
\end{enumerate}
We show that for all $\eta \in [h_1^+, h_2^-]$ the family of periodic orbits is stable, and in $[h_1^-, h_1^+]$ and $[h_2^-, h^+_2]$ we verify that one eigenvalue remains in a bounded region of the left-half plane of $\mathbb{C}$.
The procedure consists in the following steps.
\begin{enumerate}
\item
We solve \eqref{eq:floquet_ivp} for all $\eta \in [-1, 1]$ by finding a zero $(\tilde{C}, \tilde{V}) \in \mathcal{Y}_\nu$ of $G$.
To prove the existence of a zero of $G$, we rely on a contraction argument in the vein of Section \ref{sec:continuation}; in fact, since $G$ is quadratic, verifying the assumptions of Theorem \ref{thm:nk} is easier, and we postpone the details to Appendix \ref{app:contraction}.

\item
Since $0 \in [h_1^+, h_2^-]$, we verify that $\Re(\mu_1(0)) < 0$ and $\Re(\mu_2(0)) < 0$.
To this end, we apply the Gershgorin circle theorem to $\Xi^{-1} \tilde{C}(0) \Xi$, since this operator has the same spectrum as $\tilde{C}(0)$, and where $\Xi$ is an approximate (numerical) eigenbasis of $\tilde{C}(0)$ so that $\Xi^{-1} \tilde{C}(0) \Xi$ is almost diagonal.

Then, for some compact region $\Omega \subset \{ z \in \mathbb{C} \, : \, \Re(z) < 0\}$ (in practice $\Omega$ is a rectangular box found numerically), we check that $\mu_1(\eta) \in \Omega$ and $\mu_2(\eta) \in \Omega$ for all $\eta \in [h_1^+, h_2^-]$.
We accomplish this by verifying
\begin{equation}
\sup_{\eta \in [h^+_1, h^-_2]} \sup_{z \in \partial \Omega} \| \big(z I_{3 \times 3} - \tilde{\tau} (\eta) \partial_u f(\tilde{\zeta}_1 (\eta), \tilde{\zeta}_2 (\eta), \tilde{u} (\eta))\big)^{-1} \|_1 < \infty,
\end{equation}
using, in particular, Lemma \ref{lem:enclosure-order0}.

\item We repeat the strategy of Step 2 on $[\eta_1^-, \eta_1^+]$ (resp., $[\eta_2^-, \eta_2^+]$) to show that one eigenvalue of $\tilde{C}(\eta)$ remains in a compact region of $\{z \in \mathbb{C} \, : \, \Re(z) < 0 \}$ for all $\eta \in [\eta_1^-, \eta_1^+]$ (resp., $\eta \in [\eta_2^-, \eta_2^+]$).
This suffices to show that the family of periodic orbits loses its stability exactly at the crossing of $Q_1$ and $Q_2$ occurring at $\hat{\eta}_1$ and $\hat{\eta}_2$ respectively.
Indeed, only one eigenvalue can cross the imaginary axis.
Since this occurs through the origin $0$ at $Q_1$ and $Q_2$, any additional crossing of the imaginary axis must also pass through the origin $0$.
The contraction argument guarantees that $DF(\tilde{\chi})$ is invertible which, in particular, forbids the two non-trivial Floquet exponents $\mu_1$ and $\mu_2$ to vanish except on $Q_1$ and $Q_2$ -- each being crossed exactly once.
\end{enumerate}

\begin{figure}
\centering
\includegraphics[width=0.5\textwidth]{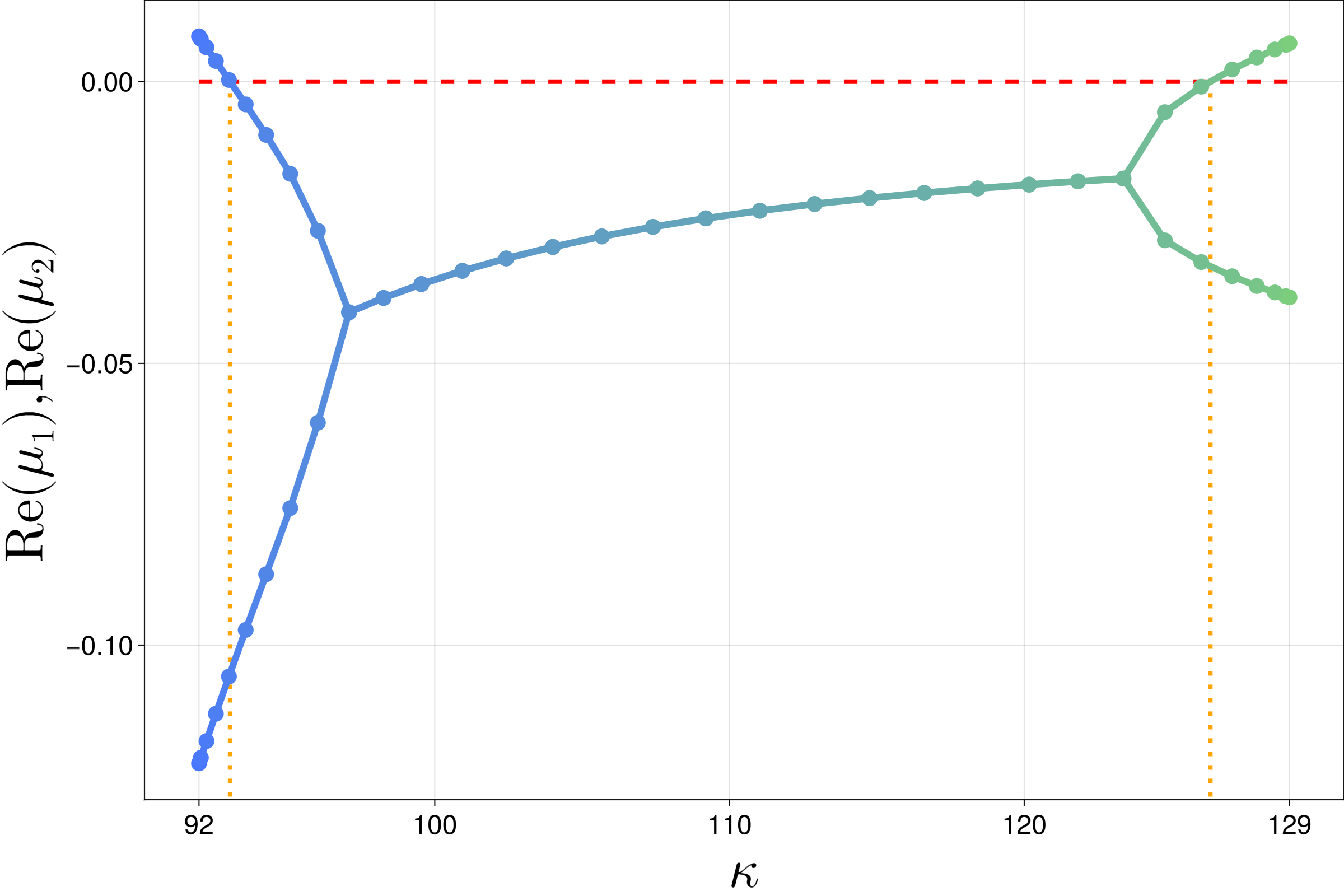}
\caption{Real part of the non-trivial Floquet exponents $\mu_1$ and $\mu_2$ associated with the global family of periodic orbits detailed in Theorem \ref{thm:main}. The dashed red line $\{\mathrm{Re}(\mu) = 0\}$ is crossed twice -- at the intersection with the vertical dotted orange lines --, corresponding to the transcritical bifurcations at the invariant boundary planes $Q_1$ and $Q_2$ around $\hat{\kappa}_1 \approx 93.0545$ and $\hat{\kappa}_2 \approx 126.3145$, respectively.
The real part of the Floquet exponents coincides when the multipliers $e^{2 \pi \mu_1}$ and $e^{2 \pi \mu_2}$ are complex conjugate.
The line width is chosen sufficiently large to encompass the error.}
\label{fig:floquet_exponents}
\end{figure}

To establish the stability of the family of periodic orbits described in Theorem \ref{thm:main}, we used a Fourier--Chebyshev approximation $\bar{\upsilon} = (\bar{C}, \bar{V}) \in \Pi_{N, K} \mathcal{Y}_\nu$ with $K = 20$, $N = 30$.
The exact solution $\tilde{\upsilon} = (\tilde{C}, \tilde{V})$ to \eqref{eq:floquet_ivp} lies within a distance $\| \tilde{\upsilon} - \bar{\upsilon} \|_{C^0} \le \| \tilde{\upsilon} - \bar{\upsilon} \|_{\mathcal{Y}_\nu} \le 5 \times 10^{-5}$; see \cite{Code}.
Figure \ref{fig:floquet_exponents} shows the non-trivial eigenvalues of $\tilde{C}$, i.e. the Floquet exponents $\mu_1$ and $\mu_2$.



\section{Outlook and future work}
\label{sec:future}

Theorem \ref{thm:main} provides a global family of positive periodic orbits in the shape of a tube, whose stable periodic orbits reach a boundary limit cycle; see Figure \ref{fig:tube}.
Nevertheless, through (typically multi-parameter) continuations, other invariant subsets or more intricate dynamics near our global family could possibly occur.
As main examples, periodic orbits may undergo bifurcations such as a Neimark--Sacker, a period-doubling, a homoclinic bifurcation, or a blue sky catastrophe.
See for instance \cite[Chapter 5--7]{Kuznetsov2023} and the references therein.
As far as we know, there is no evidence of Neimark--Sacker bifurcation giving rise to invariant tori, or blue-sky catastrophes; but there is evidence of a cascade of period-doubling bifurcations; see \cite{HsuChaos, Kryz21}.

As a matter of fact, using both $\kappa$ and $a_1$ as adjustable parameters, the proven global family at $a_1 = 10$ (see Figure \ref{fig:tube}) seems to eventually lose its stability in the interior when $a_1$ decreases.
We did not attempt to use our continuation method to rigorously track the value of $a_1$ at which the loss of stability occurs, as the global branch becomes increasingly challenging to approximate.

 As a visual cue, Figure \ref{fig:tube_a1_6} shows that the periodic orbits transition from $Q_1$ to $Q_2$ quite abruptly.
This discrepancy between the two boundary planes ``kinks'' the tube of solutions, as the Fourier--Chebyshev coefficients of the solution decay much slower.
Thus, achieving a proof in this regime requires a substantially larger number of Fourier modes $K$ and Chebyshev modes $N$.
While the use of fast Fourier transform algorithms mitigate the computational complexity with respect to the parameter (essentially, $O(N\log N)$ cost at each call site), we still need to deal with large complexity from constructing the linear operators $A$ (i.e., $O(K^3)$ cost to invert $N+1$ matrices) and performing matrix products with a truncation of $DF(\bar{\chi})$ to estimate $Z_1$.
The last cost is exacerbated by the use of interval arithmetic, with sufficiently tight enclosures to avoid loss of accuracy.
For the proof of stability, naturally, these computational considerations must likewise be taken into account.

Numerical evidence suggests that this instability arises from the occurrence of period-doubling bifurcations along the branch; see Figure \ref{fig:period_doubling}.
It may be the case that further decreasing $a_1$ results in a cascade of period-doubling bifurcations, eventually leading to a chaotic attractor in $\mathbb{R}_+^3$ around $a_1 \approx 4$; see Figure \ref{fig:attractor}.
If there is indeed a chaotic attractor, its existence, and the mechanism behind its birth deserve further study.

{While this article focused on a single Lotka--Volterra model \eqref{eq:original_model}, the CAP-based continuation framework we have presented is not limited to this example. In the future, it can be adapted and automated to explore broader families of Lotka--Volterra systems \eqref{eq:general_model} with varying dimensions and interaction structures. This opens the possibility of systematically investigating mechanisms of stable coexistence and exclusion across models that were previously intractable.
By providing a rigorous computational approach, our work lays a foundation for a more comprehensive understanding of resource-competition dynamics in complex ecological systems.}

\begin{figure}[H]
\centering
\begin{subfigure}[b]{0.25\textwidth}
\includegraphics[width=\textwidth]{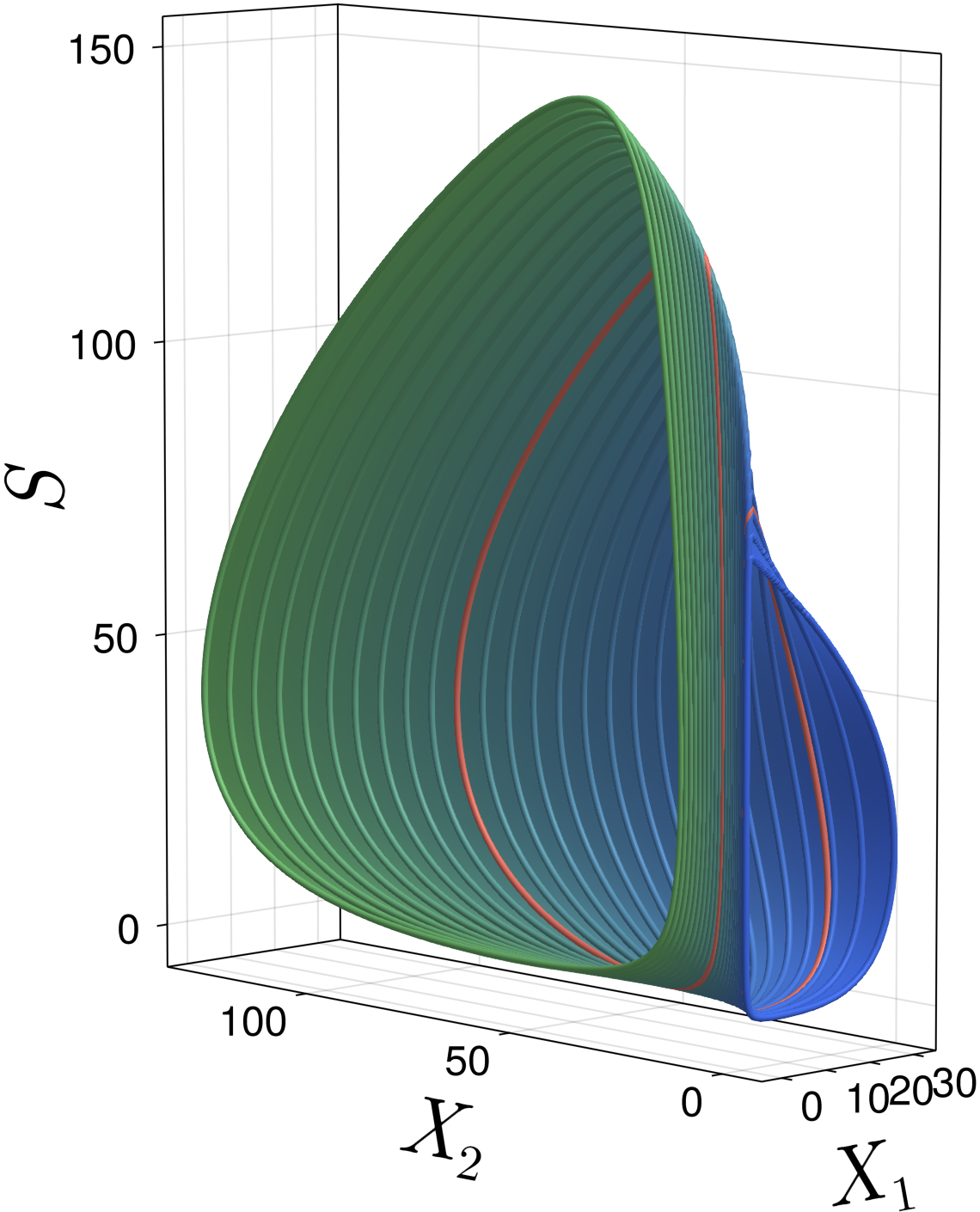}
\caption{Numerical global family.}
\label{fig:tube_a1_6}
\end{subfigure}
\hspace{0.02\textwidth}
\begin{subfigure}[b]{0.37\textwidth}
\includegraphics[width=\textwidth]{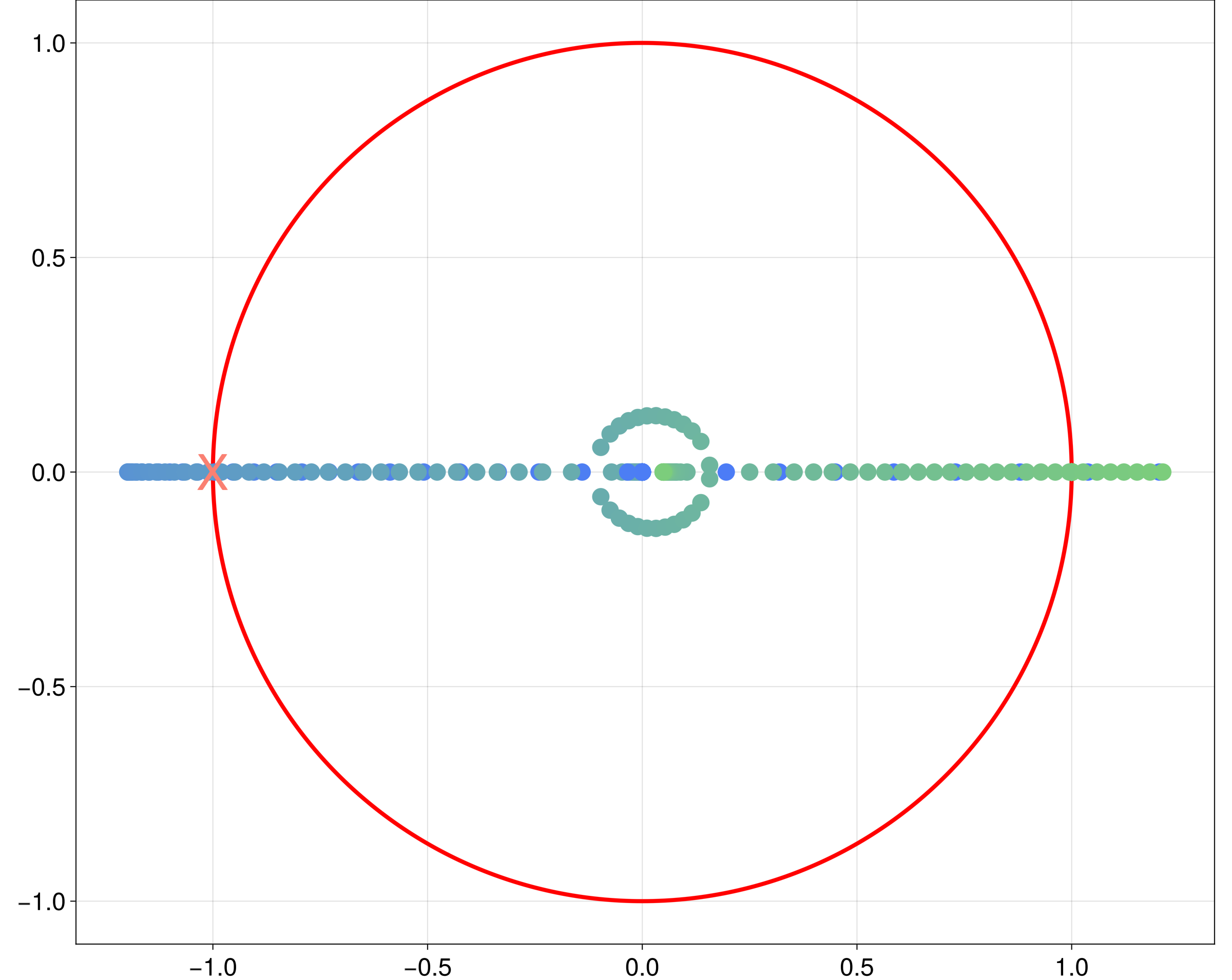}
\caption{Numerical Floquet multipliers.}
\label{fig:period_doubling}
\end{subfigure}
\hspace{0.02\textwidth}
\begin{subfigure}[b]{0.3\textwidth}
\includegraphics[width=\textwidth]{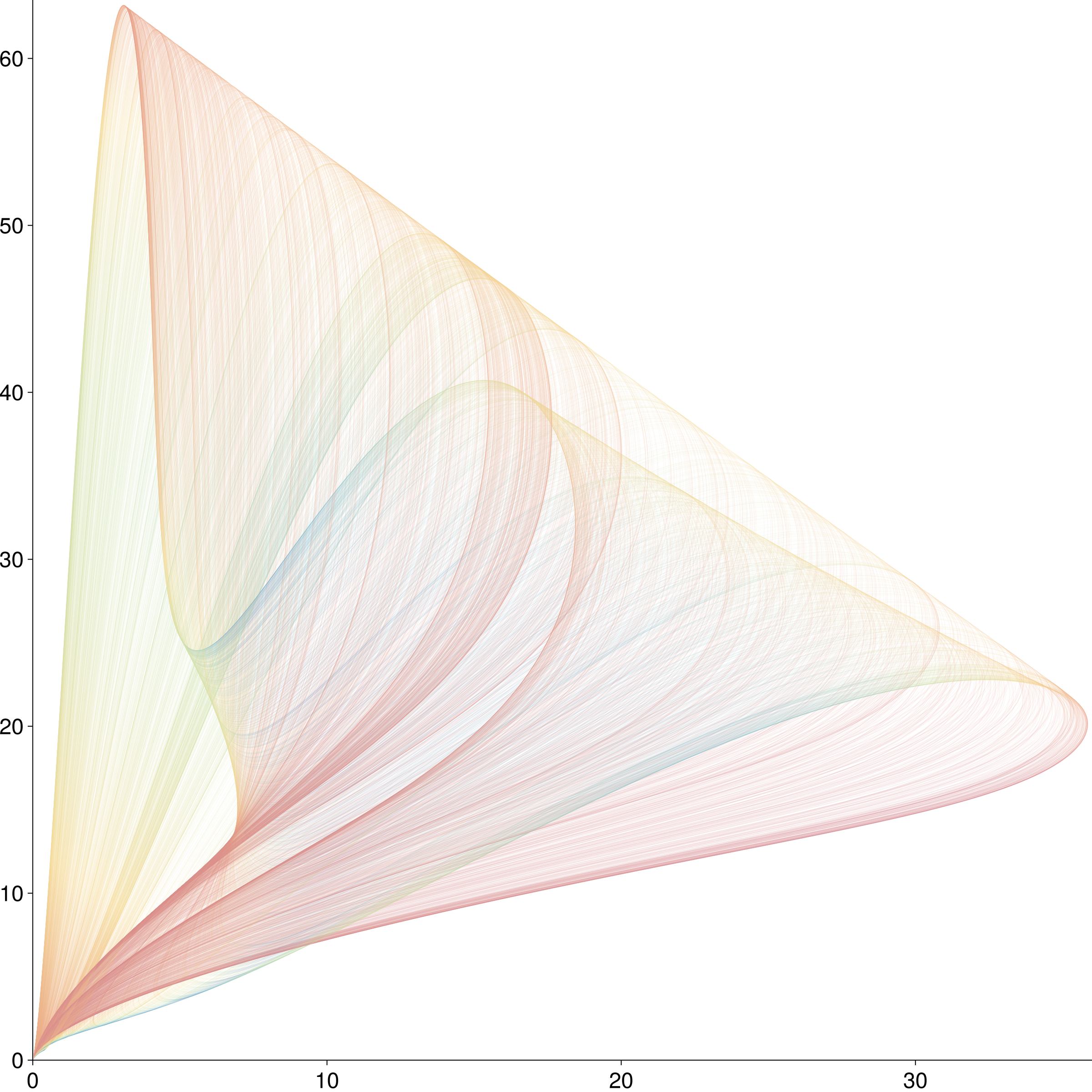}
\caption{Numerical chaotic attractor.}
\label{fig:attractor}
\end{subfigure}
\caption{The parameter values are $a_2 = 41$, $d_1 = 0.8$, $d_2 = 0.5$, and $m_1 = m_2 = y_1 = y_2 = \gamma = 1$.
(a)~Numerical approximation of the global family for $a_1 = 6$. The orange rings corresponds to a numerical observation of a period-doubling bifurcation.
(b) Numerical approximation of the Floquet multipliers associated to (a). The orange cross marks the crossing through $-1$, i.e., a potential period-doubling bifurcation.
(c) Projection onto the $(X_1, X_2)$-plane of a numerical chaotic attractor in $\mathbb{R}^3_+$ for $a_1 = 4$.}
\end{figure}

\section{Acknowledgments}

We are grateful to Sze-Bi Hsu for the suggestion of solving the stable connection problem and to Bernold Fiedler for many inspiring discussions.
J.-Y.~Dai was supported by NSTC (National Science and Technology Council) under grant No. 113-2628-M-007-005-MY4.
O.~H\'{e}not was supported by ANR under project No. CAPPS: ANR-23-CE40-0004-01, and by NSTC under grant No. 115-2115-M-002-001-MY2.
O.~Hénot also thanks the National Center for Theoretical Sciences for its support during the preparation of this work.
P.~Lappicy was supported by Marie Skłodowska--Curie Actions, UNA4CAREER H2020 Cofund, 847635, with the project DYNCOSMOS.
N.~Vassena was supported by the DFG (German Research Society), project No. 512355535.

\appendix

\section{Contraction argument for the Floquet normal form}
\label{app:contraction}

The purpose of this appendix is to provide some insight on the rigorous computation of the Floquet normal form of the fundamental matrix solution to \eqref{eq:floquet_ivp}.
This is a variation of the approach presented in \cite{FloquetNormalFormCAP}, adapted to work with our continuation method.

In the context of this manuscript, we want to verify Theorem~\ref{thm:nk} in the Banach space $\mathcal{Y}_\nu$ and the fixed-point operator
\begin{equation}
\upsilon \quad \mapsto \quad \upsilon - B G(\upsilon),
\end{equation}
thereby obtaining a zero of the mapping $G$ defined in \eqref{eq:G}.
Similarly to the construction of the injective linear operator $A$, we set $B \in \mathscr{B}(\mathcal{Y}_\nu, \mathcal{Y}_\nu)$ as
\begin{equation}
B \bydef B_\textnormal{finite} \Pi_K + B_\textnormal{tail} \Pi_{> K},
\end{equation}
with $B_\textnormal{finite} \in \mathscr{B}(\Pi_K \mathcal{Y}_\nu, \Pi_K \mathcal{Y}_\nu)$ an approximation of $(\Pi_{N,K} DG(\bar{\upsilon}) \Pi_{N,K})^{-1}$ and $B_\textnormal{tail} : \Pi_{> K} \mathcal{Y}_\nu \to \Pi_{> K} \mathcal{Y}_\nu$ is defined, for all $\upsilon = (C, V) \in \mathcal{Y}_\nu$, by
\begin{equation}
B_\textnormal{tail} \upsilon \bydef (0, I_{3\times 3} \otimes V'), \qquad
(V_{ij}')_k \bydef
\begin{cases}
0, & |k| \le K, \\
(ik)^{-1} (V_{ij})_k, & |k| > K,
\end{cases} \quad i,j = 1, 2, 3.
\end{equation}
The symbol $\otimes$ denotes the Kronecker product between 3-by-3 matrices; this operation comes out from matrix differentiation.
In particular, for $\upsilon = (C, V) \in \mathcal{Y}_\nu$,
\begin{equation*}
DG(\upsilon) =
\begin{pmatrix}
0 & I_{3\times 3} \otimes E_0 \Pi_{N,K} \\
I_{3\times 3} \otimes V & I_{3\times 3} \otimes \partial_t + C^T \otimes I_{3\times 3} - I_{3\times 3} \otimes \tilde{\tau} \partial_u f(\tilde{\zeta}_1, \tilde{\zeta}_2, \tilde{u})
\end{pmatrix},
\end{equation*}
where $C^T$ is the 3-by-3 matrix transpose of $C$ and $E_0 : \textnormal{Mat}_{3\times 3} (\mathcal{U}_\nu) \to \textnormal{Mat}_{3\times 3} (\mathcal{P}_\nu)$ is the evaluation operator at $0$, i.e. $E_0 V \bydef V(0)$.

{We apply Theorem~\ref{thm:nk}.}
Formulas for $\|BG(\bar{\upsilon})\|_{\mathcal{Y}_\nu}$, $\|BDG(\bar{\upsilon}) - I\|_{\mathscr{B}(\mathcal{Y}_\nu, \mathcal{Y}_\nu)}$, and $\sup_{\upsilon \in B_R(\bar{\upsilon})} \|BD^2G(\upsilon)\|_{\mathscr{B}(\mathcal{Y}_\nu, \mathcal{Y}_\nu)}$ are in fact much simpler than those obtained for $F$ in Sections \ref{sec:Ybound}--\ref{sec:Z2bound}.
Indeed, $G$ has only the single quadratic term $V C$.
Thus, we can freely pick $R = \infty$, and $\|BD^2G(\upsilon)\|_{\mathscr{B}(\mathcal{Y}_\nu, \mathcal{Y}_\nu)} \le 2 \| B \|_{\mathscr{B}(\mathcal{Y}_\nu, \mathcal{Y}_\nu)}$ for all $\upsilon \in \mathcal{Y}_\nu$.

By construction, $\|B\|_{\mathscr{B}(\mathcal{Y}_\nu, \mathcal{Y}_\nu)} = \max\Big(\|B_\textnormal{finite}\|_{\mathscr{B}(\mathcal{Y}_\nu, \mathcal{Y}_\nu)}, \frac{1}{K+1}\Big)$, and it remains two control $\| B G (\bar{\upsilon}) \|_{\mathcal{Y}_\nu}$ and $\|B DG (\bar{\upsilon}) - I\|_{\mathscr{B}(\mathcal{Y}_\nu, \mathcal{Y}_\nu)}$.
In regards to $\| B G (\bar{\upsilon}) \|_{\mathcal{Y}_\nu}$, we simply use the triangle inequality
\[
\| B G (\bar{\upsilon}) \|_{\mathcal{Y}_\nu}
\le \| B \|_{\mathscr{B}(\mathcal{Y}_\nu, \mathcal{Y}_\nu)} \| G (\bar{\upsilon}) \|_{\mathcal{Y}_\nu},
\]
where computing $G (\bar{\upsilon})$ amounts to being able to calculate $\tilde{\tau} \partial_u f(\tilde{u}, \tilde{\zeta}_1, \tilde{\zeta}_2) \bar{V}$ rigorously which is done by means of Lemma \ref{lem:Z2_inv}.
At last, define
\begin{equation}
\hat{W}_1 \bydef
\begin{pmatrix}
0 & I_{3\times 3} \otimes E_0 \Pi_{N,K} \\
I_{3\times 3} \otimes \bar{V} & I_{3\times 3} \otimes \partial_t + \bar{C}^T \otimes I_{3\times 3} - \omega_1
\end{pmatrix},
\end{equation}
where $\omega_1$ is a banded operator introduced in Section \ref{sec:Z1bound} to approximate $\bar{\tau} \partial_u f(\bar{u}, \bar{\zeta}_1, \bar{\zeta}_2)$.
Then,
\begin{align*}
\| B DG (\bar{\upsilon}) - I\|_{\mathscr{B}(\mathcal{Y}_\nu, \mathcal{Y}_\nu)}
&\le \| B \hat{W}_1 - I \|_{\mathscr{B}(\mathcal{Y}_\nu, \mathcal{Y}_\nu)} + \| B \|_{\mathscr{B}(\mathcal{Y}_\nu, \mathcal{Y}_\nu)} \| DG(\bar{\upsilon}) - \hat{W}_1 \|_{\mathscr{B}(\mathcal{Y}_\nu, \mathcal{Y}_\nu)},
\end{align*}
with
\begin{align*}
\| DG(\bar{\upsilon}) - \hat{W}_1 \|_{\mathscr{B}(\mathcal{Y}_\nu, \mathcal{Y}_\nu)}
& = \| \tilde{\tau} \partial_u f (\tilde{\zeta}_1, \tilde{\zeta}_2, \tilde{u}) - \omega_1 \|_{\mathscr{B}(\mathcal{U}_\nu, \mathcal{U}_\nu)}
\\&
\le \max_{1 \le j \le 3} \sum_{i=1}^3 \| \tilde{\tau} \partial_{u_j} f_i (\tilde{\zeta}_1, \tilde{\zeta}_2, \tilde{u}) - \omega_1^{(i,j)} \|_\nu,
\end{align*}
and, motivated by the banded structure of $\hat{W}_1$, we write
\begin{equation} \label{eq-81}
\| B \hat{W}_1 - I \|_{\mathscr{B}(\mathcal{Y}_\nu, \mathcal{Y}_\nu)}
= \max\left(
\| B \hat{W}_1 \Pi_{2K} - \Pi_{2K} \|_{\mathscr{B}(\mathcal{Y}_\nu, \mathcal{Y}_\nu)}, \| B \hat{W}_1 \Pi_{> 2K} - \Pi_{> 2K} \|_{\mathscr{B}(\mathcal{Y}_\nu, \mathcal{Y}_\nu)}
\right).
\end{equation}
For the first term in the right-hand side of \eqref{eq-81}, we have
\begin{align*}
&\|B \hat{W}_1 \Pi_{2K} - \Pi_{2K}\|_{\mathscr{B}(\mathcal{Y}_\nu, \mathcal{Y}_\nu)} \\
&= \|B_\textnormal{finite} \Pi_K \hat{W}_1 \Pi_{2K} + B_\textnormal{tail} \Pi_{> K} W_1 \Pi_{2K} - \Pi_{2K}\|_{\mathscr{B}(\mathcal{Y}_\nu, \mathcal{Y}_\nu)} \\
&\le \| B_\textnormal{finite} \Pi_K \hat{W}_1 \Pi_{2K} - \Pi_K \|_{\mathscr{B}(\mathcal{Y}_\nu, \mathcal{Y}_\nu)} + \|B_\textnormal{tail} \Pi_{> K} W_1 \Pi_{2K} + \Pi_K - \Pi_{2K}\|_{\mathscr{B}(\mathcal{Y}_\nu, \mathcal{Y}_\nu)} \\
&\le \| B_\textnormal{finite} \Pi_K \hat{W}_1 \Pi_{2K} - \Pi_K \|_{\mathscr{B}(\mathcal{Y}_\nu, \mathcal{Y}_\nu)} + \frac{1}{K + 1} \|\Pi_{> K} \omega_1 \Pi_{2K} + \Pi_{> K} \Pi_{2K} + \Pi_K - \Pi_{2K}\|_{\mathscr{B}(\mathcal{U}_\nu, \mathcal{U}_\nu)} \\
&\le \| B_\textnormal{finite} \Pi_K \hat{W}_1 \Pi_{2K} - \Pi_K \|_{\mathscr{B}(\mathcal{Y}_\nu, \mathcal{Y}_\nu)} + \frac{1}{K + 1} \max_{1 \le j \le 3} \sum_{i = 1}^3 \|\omega_1^{(i,j)}\|_\nu.
\end{align*}
For the second one, we find
\begin{align*}
\|B \hat{W}_1 \Pi_{> 2K} - \Pi_{> 2K}\|_{\mathscr{B}(\mathcal{Y}_\nu, \mathcal{Y}_\nu)}
&\le \|B_\textnormal{finite}\|_{\mathscr{B}(\mathcal{Y}_\nu, \mathcal{Y}_\nu)} \|\Pi_{K} \omega_1 \Pi_{> 2K}\|_{\mathscr{B}(\mathcal{U}_\nu, \mathcal{U}_\nu)} \\
&\qquad+ \|B_\textnormal{tail}\|_{\mathscr{B}(\mathcal{Y}_\nu, \mathcal{Y}_\nu)} \|\Pi_{> K} \omega_1 \Pi_{> 2K}\|_{\mathscr{B}(\mathcal{U}_\nu, \mathcal{U}_\nu)} \\
&\le \frac{1}{K+1} \max_{1 \le j \le 3} \sum_{i = 1}^3 \|\omega_1^{(i,j)}\|_\nu.
\end{align*}

\bibliographystyle{abbrv}
\bibliography{references}

\end{document}